\documentclass[11pt,a4paper]{article}
\usepackage[margin=1.0truein]{geometry}

\usepackage{setspace}   
\usepackage{pdflscape}  
\usepackage{microtype}  
\usepackage{xcolor}     

\usepackage{amsmath,amssymb,amsthm} 
\usepackage{mathtools}              
\usepackage{mathrsfs}               
\usepackage{latexsym}               
\usepackage{bm}                     

\usepackage{wrapfig}                          
\usepackage[subrefformat=parens]{subcaption} 

\usepackage{array,booktabs,multirow,diagbox,threeparttable,siunitx,tabularx}
\newcolumntype{P}[1]{>{\centering\arraybackslash}p{#1}} 
\newcolumntype{C}{>{\centering\arraybackslash}X}        
\usepackage{pbox} 

\usepackage{algorithm}
\usepackage[noend]{algpseudocode}

\algnewcommand{\Break}{\textbf{break}}
\algnewcommand{\algorithmicgoto}{\textbf{go to} Line}
\algnewcommand{\Goto}[1]{\algorithmicgoto~\ref{#1}}
\algblockdefx{MRepeat}{EndRepeat}{\textbf{repeat}}{}
\algnotext{EndRepeat}

\usepackage[sort&compress,numbers,square]{natbib}

\usepackage{thmtools,thm-restate}

\usepackage{url}       
\usepackage{multicol}  
\usepackage{wasysym}   
\usepackage{adjustbox} 
\usepackage{xparse}    
\usepackage{authblk}   
\usepackage{xspace}    
\usepackage{enumitem}  

\usepackage{hyperref}
\usepackage{cleveref}

\expandafter\def\csname ver@etex.sty\endcsname{3000/12/31}

\usepackage{autonum} 

\makeatletter
\autonum@generatePatchedReferenceCSL{Cref} 
\makeatother

\hypersetup{
  colorlinks,
  linkcolor={green!35!black},
  citecolor={green!35!black},
  urlcolor={blue!50!black}
}

\declaretheoremstyle[
  shaded={bgcolor=gray!15},
]{thmsty}

\declaretheorem[
  name=Theorem,
  refname={Theorem,Theorems},
  style=thmsty,
]{theorem}

\declaretheorem[
  name=Corollary,
  refname={Corollary,Corollaries},
  style=thmsty,
]{corollary}

\declaretheorem[
  name=Lemma,
  refname={Lemma,Lemmas},
  style=thmsty,
]{lemma}

\declaretheorem[
  name=Assumption,
  refname={Assumption,Assumptions},
  style=thmsty,
]{assumption}

\crefname{algorithm}{Algorithm}{Algorithms}
\crefname{line}{Line}{Lines}
\crefname{section}{Section}{Sections}
\crefname{appendix}{Appendix}{Appendices}
\crefname{table}{Table}{Tables}
\crefname{figure}{Figure}{Figures}

\crefname{equation}{}{}
\Crefname{equation}{Eq.}{Eqs.}

\setlist[itemize]{
  topsep=0.1\baselineskip,
  itemsep=0\baselineskip,
  leftmargin=1.5em,
}

\setlist[enumerate]{
  font=\upshape,
  label=(\alph*),
  ref=(\alph*),
  topsep=0.1\baselineskip,
  itemsep=0\baselineskip,
  leftmargin=2em,
}

\newlist{enuminasm}{enumerate}{1}
\setlist[enuminasm]{
  font=\upshape,
  label=(\alph*),
  ref=\theassumption(\alph*),
  topsep=0.4\baselineskip,
  itemsep=0\baselineskip,
  leftmargin=2em,
}
\crefalias{enuminasmi}{assumption}

\newlist{enuminthm}{enumerate}{1}
\setlist[enuminthm]{
  font=\upshape,
  label=(\alph*),
  ref=\thetheorem(\alph*),
  topsep=0.4\baselineskip,
  itemsep=0\baselineskip,
  leftmargin=2em,
}
\crefalias{enuminthmi}{theorem}

\newlist{enuminlem}{enumerate}{1}
\setlist[enuminlem]{
  font=\upshape,
  label=(\alph*),
  ref=\thelemma(\alph*),
  topsep=0.4\baselineskip,
  itemsep=0\baselineskip,
  leftmargin=2em,
}
\crefalias{enuminlemi}{lemma}

\newlist{enuminprop}{enumerate}{1}
\setlist[enuminprop]{
  font=\upshape,
  label=(\alph*),
  ref=\theproposition(\alph*),
  topsep=0.4\baselineskip,
  itemsep=0\baselineskip,
  leftmargin=2em,
}
\crefalias{enuminpropi}{proposition}

\newlist{enumincond}{enumerate}{1}
\setlist[enumincond]{
  font=\upshape,
  label=(\alph*),
  ref=\thecondition(\alph*),
  topsep=0.4\baselineskip,
  itemsep=0\baselineskip,
  leftmargin=2em,
}
\crefalias{enumincondi}{condition}

\DeclareMathOperator*{\argmin}{argmin}

\DeclareMathOperator{\EE}{\mathbb{E}}

\DeclarePairedDelimiter\ceil{\lceil}{\rceil}

\DeclarePairedDelimiterX\inner[2]{\langle}{\rangle}{{#1},{#2}}
\DeclarePairedDelimiter\abs{\lvert}{\rvert}
\DeclarePairedDelimiter\norm{\lVert}{\rVert}

\providecommand\given{}
\newcommand\SetSymbol[1][]{%
  \nonscript\:#1\vert
  \allowbreak
  \nonscript\:
  \mathopen{}%
}

\DeclarePairedDelimiter\prn{(}{)}
\DeclarePairedDelimiter\brk{[}{]}
\DeclarePairedDelimiter\set{\lbrace}{\rbrace}

\DeclarePairedDelimiterX\Prn[2]{(}{)}{%
  \renewcommand\given{\SetSymbol[\delimsize]}%
  #1\given#2%
}

\DeclarePairedDelimiterX\Brk[2]{[}{]}{%
  \renewcommand\given{\SetSymbol[\delimsize]}%
  #1\given#2%
}

\DeclarePairedDelimiterX\Set[2]{\lbrace}{\rbrace}{%
  \renewcommand\given{\SetSymbol[\delimsize]}%
  #1\given#2%
}

\newcommand{\N}{\mathbb N}

\newcommand{\R}{\mathbb R}

\newcommand{\0}{\mathbf 0}

\renewcommand{\O}{\mathrm{O}}

\renewcommand{\epsilon}{\varepsilon}

\NewDocumentCommand{\ex}{s o d<> m}{%
  \IfBooleanTF{#1}{%
    \IfValueTF{#3}{\EE_{#3}\brk*{#4}}{\EE\brk*{#4}}%
  }{%
    \IfValueTF{#3}{%
      \IfNoValueTF{#2}{\EE_{#3}\brk{#4}}{\EE_{#3}\brk[#2]{#4}}%
    }{%
      \IfNoValueTF{#2}{\EE\brk{#4}}{\EE\brk[#2]{#4}}%
    }%
  }%
}

\NewDocumentCommand{\cex}{s o d<> m m}{%
  \IfBooleanTF{#1}{%
    \IfValueTF{#3}{\EE_{#3}\Brk*{#4}{#5}}{\EE\Brk*{#4}{#5}}%
  }{%
    \IfValueTF{#3}{%
      \IfNoValueTF{#2}{\EE_{#3}\Brk{#4}{#5}}{\EE_{#3}\Brk[#2]{#4}{#5}}%
    }{%
      \IfNoValueTF{#2}{\EE\Brk{#4}{#5}}{\EE\Brk[#2]{#4}{#5}}%
    }%
  }%
}

\usepackage{CJKutf8}

\newcommand{\by}[2][]{\text{\pbox[c]{\textwidth}{(by \pbox[t]{\textwidth}{\phantom{}#2)#1}}}}
\newcommand{\email}[1]{\href{mailto:#1}{\nolinkurl{#1}}}

\date{\today\vspace{-1\baselineskip}}

\author[1]{Naoki Marumo\thanks{E-mail: \email{marumo@mist.i.u-tokyo.ac.jp}}}
\author[1,2]{Akiko Takeda}

\affil[1]{\normalsize Graduate School of Information Science and Technology, University of Tokyo, Tokyo, Japan}
\affil[2]{\normalsize RIKEN Center for Advanced Intelligence Project, Tokyo, Japan}

\title{A General Recipe for Parameter-Free Nonconvex Optimization via Higher-Order Regularization}

\begin{document}
\maketitle

\begin{abstract}
  We develop a systematic framework for constructing parameter-free algorithms for smooth nonconvex optimization.
  The framework is based on higher-order regularization: each step is computed from a regularized local model whose regularization exponent exceeds the order of the model error.
  This design makes the resulting method robust to misspecification of the regularization parameter and yields complexity bounds without backtracking or other acceptance tests.
  We apply the framework to gradient descent, Newton's method, the Gauss--Newton method, stochastic gradient descent, and PAGE.
  Without prior knowledge of problem-dependent parameters, the resulting algorithms achieve complexity bounds with optimal or best-known dependence on the target accuracy.
  When the problem-dependent parameters are known up to constant factors, suitable tuning also recovers the optimal or best-known dependence on those parameters.
\end{abstract}

\smallskip
\begin{description}[leftmargin=!,labelwidth=\widthof{\bfseries Keywords:}]
  \item[Keywords:]
  Nonconvex optimization,
  parameter-free algorithms,
  model-based methods,
  oracle complexity,
  stochastic optimization,
  variance reduction
  \item[MSC2020:]
  90C26, 90C30, 65K05, 49M15
\end{description}
\smallskip

\section{Introduction}
We consider the nonconvex optimization problem
\begin{align}
  \min_{x \in \R^d}\ f(x),
\end{align}
where $f \colon \R^d \to \R$ is smooth and bounded below.
For such problems, oracle-based algorithms, which access $f$ only through oracles such as evaluations of $f$, $\nabla f$, and $\nabla^2 f$, have been extensively studied.
A central performance measure is the \emph{oracle complexity}~\citep{nemirovsky1983problem}, or evaluation complexity~\citep{cartis2022evaluation}, defined as the number of oracle calls required to find an $\epsilon$-stationary point, i.e., a point $x$ satisfying $\norm*{\nabla f(x)} \leq \epsilon$.
Under standard assumptions, gradient descent achieves the complexity $\O(\epsilon^{-2})$ \citep[Section 1.2.3]{nesterov2004introductory}, while the cubic-regularized Newton method achieves $\O(\epsilon^{-3/2})$ \citep{nesterov2006cubic}.
One straightforward way to obtain these complexity bounds is to set algorithmic parameters (e.g., stepsizes and regularization parameters) in terms of problem-dependent parameters (e.g., the Lipschitz constants of $\nabla f$ and $\nabla^2 f$).

\emph{Parameter-free algorithms} seek to obtain these complexity bounds without requiring problem-dependent parameters as input.
This matters in practice because such parameters are often unavailable a priori.
A classical way to design parameter-free algorithms is to adapt algorithmic parameters until a trial point passes an acceptance test, such as a sufficient-decrease condition on the objective value.
This idea underlies line-search, trust-region, and adaptive-regularization methods~\citep{nocedal2006numerical,conn2000trust,cartis2011adaptive,cartis2011adaptive2}.
Because such tests can be attached to different local models, this approach applies broadly to model-based algorithms, including Newton's method and the Gauss--Newton method.
However, evaluating these tests can be costly, unreliable, or unavailable, especially in stochastic optimization.

Parameter-free techniques that avoid acceptance tests have mostly been developed for first-order methods.
One line of work estimates local smoothness from quantities computed during the run, such as gradient differences or estimates of model error, and chooses the stepsize accordingly~\citep{malitsky2020adaptive,ye2025simple,yagishita2025simple}.
Another line chooses stepsizes using the history of gradient norms, typically scaling the stepsize inversely with an accumulated gradient norm~\citep{li2019convergence,attia2023sgd,khaled2024tuning,ward2020adagrad,faw2022power,gratton2026complexity}.
Some of these methods are compatible with stochastic settings; in particular, AdaGrad-Norm~\citep{ward2020adagrad,faw2022power} has been successfully analyzed in stochastic nonconvex optimization.
However, they remain essentially first-order, and it is less clear how to extend them systematically to model-based methods that use Hessian, Jacobian, or other local information.

Thus, existing parameter-free mechanisms exhibit a tradeoff between two desirable properties: broad applicability to model-based algorithms and compatibility with stochastic optimization.
This motivates the search for a mechanism that combines these two advantages.

\paragraph{Contributions.}
We develop a systematic framework for constructing parameter-free algorithms that combines these two advantages.
The key idea is higher-order regularization.
Subproblem-based methods often include a regularization term to control the step computed from a local model.
Instead of matching the regularization exponent to the order of the local-model error, we choose it to be strictly larger.
This higher-order regularization makes the method robust to misspecification of the regularization parameter, a key difficulty when problem-dependent parameters are unknown.
This allows us to fix a regularization-parameter schedule in advance.

The resulting framework has the following features.
\begin{itemize}
  \item
  The framework applies broadly to algorithms that compute the next iterate by solving a subproblem constructed from a local approximation of $f$ around the current iterate.
  This class covers a wide range of methods, including gradient descent, Newton's method, and the Gauss--Newton method.
  \item
  The resulting parameter-free algorithms involve no backtracking or other acceptance tests.
  Thus, they avoid the additional cost of validating trial points, which makes them amenable to stochastic settings.
\end{itemize}

To demonstrate the generality of this approach, we apply it to several deterministic and stochastic algorithms.
Specifically, we analyze higher-order regularized variants of gradient descent, Newton's method, the Gauss--Newton method, stochastic gradient descent (SGD)~\citep{robbins1951stochastic,ghadimi2013stochastic}, and PAGE~\citep{li2021page}.
These instances lead to the following guarantees.
\begin{itemize}
  \item
  Without prior knowledge of problem-dependent parameters, they achieve the optimal or best-known dependence on $\epsilon$ in their complexity bounds.
  \item
  If the problem-dependent parameters are known up to constant factors, then choosing the algorithmic parameters accordingly recovers the optimal or best-known dependence on these parameters as well as on $\epsilon$.
\end{itemize}

A related idea can be found in \citep{marumo2025parameter}, which studies a parameter-free quasi-Newton-type method with quartic regularization.
The present work can be viewed as extending this idea to a general model-based framework and applying it systematically to deterministic and stochastic algorithms.

\paragraph{Organization.}
The rest of this paper is organized as follows.
\Cref{sec:recipe} presents the general higher-order-regularization recipe for constructing parameter-free algorithms, using gradient descent as a guiding example and deriving its complexity bounds.
In \cref{sec:newton,sec:gn,sec:sgd,sec:page}, we instantiate this recipe for four additional algorithms and establish their complexity bounds.
Each of these sections also reviews related work on the corresponding algorithm.
Finally, \cref{sec:conclusion} concludes the paper.

\paragraph{Notation.}
Let $\N \coloneqq \set{0, 1, 2, \ldots}$.
Let $\R^d$ be the $d$-dimensional Euclidean space equipped with the standard inner product $\inner*{\cdot}{\cdot}$ and the induced norm $\norm*{\cdot}$.
We also use $\norm*{\cdot}$ to denote the operator norm of a matrix.
Throughout the paper, $(x_k)_{k \in \N}$ denotes the sequence of iterates generated by the algorithm under consideration.
Define
\begin{align}
  \Delta \coloneqq f(x_0) - \inf_{x \in \R^d} f(x),
  \label{eq:def_delta}
\end{align}
which is finite because $f$ is bounded below.

\section{Higher-order regularization recipe}
\label{sec:recipe}
This section explains our recipe for constructing parameter-free algorithms.
We use gradient descent as a guiding example to illustrate the common analysis pattern for the resulting algorithms.

Let $(x_k)_{k \in \N}$ be the sequence of iterates.
We focus on model-based methods that compute the step $s_k \coloneqq x_{k+1} - x_k$ by solving the subproblem
\begin{align}
  s_k
  =
  \argmin_{s \in \R^d}
  \set*{
    \tilde m_k(s)
    + \frac{\lambda_k}{p} \norm*{s}^p
  },
  \label{eq:subproblem_hor}
\end{align}
where $p > 1$, $\lambda_k > 0$, and $\tilde m_k(s)$ is a local approximation of $f(x_k + s)$.
Typically, the exponent $p$ is chosen to match the order of the model error $\abs*{\tilde m_k(s) - f(x_k + s)}$.
Our recipe instead uses an exponent strictly larger than that order.
This higher-order regularization provides a simple principle for making model-based methods parameter-free.

The complexity analyses for the resulting methods follow a common pattern.
Because this pattern is easiest to understand in a concrete setting, we demonstrate it for gradient descent.
We then summarize the analysis pattern for model-based methods of the form \cref{eq:subproblem_hor}.

\subsection{Preliminaries: Useful inequalities}
Before turning to the gradient-descent example, we collect several inequalities that will be used repeatedly in the complexity analyses.

The first inequality is an elementary upper bound for a difference of two power functions.
\begin{lemma}
  \label{lem:power_diff_bound}
  For all $a, b > 0$, $p > q \geq 1$, and $t \geq 0$, we have
  \begin{align}
    b t^q - a t^p
    \leq
    b \prn*{\frac{b}{a}}^{\frac{q}{p-q}}.
  \end{align}
\end{lemma}
\begin{proof}
  Let $\phi(t) \coloneqq b t^q - a t^p$.
  Since $\phi'(t) = b q t^{q-1} - a p t^{p-1}$, the function $\phi$ is maximized at $t = \prn[\big]{\frac{b q}{a p}}^{\frac{1}{p-q}}$.
  Hence,
  \begin{align}
    \phi(t)
    \leq
    \phi \prn*{\prn[\Big]{\frac{b q}{a p}}^{\frac{1}{p-q}}}
    =
    \frac{p-q}{p} b \prn*{\frac{b q}{a p}}^{\frac{q}{p-q}}
    \leq
    b \prn*{\frac{b}{a}}^{\frac{q}{p-q}},
  \end{align}
  where the last inequality follows from $p > q$.
\end{proof}

We will also use the following forms of H\"older's inequality.
\begin{lemma}[H\"older's inequality]
  Let $(a_i)_{i=1}^n$, $(b_i)_{i=1}^n$, and $(c_i)_{i=1}^n$ be nonnegative sequences.
  For all $\alpha, \beta > 0$ such that $\alpha + \beta = 1$, we have
  \begin{align}
    \sum_{i=1}^n a_i^\alpha b_i^\beta
    &\leq
    \prn*{
      \sum_{i=1}^n a_i
    }^\alpha
    \prn*{
      \sum_{i=1}^n b_i
    }^\beta.
    \label{eq:holder_2}
  \end{align}
  For all $\alpha, \beta, \gamma > 0$ such that $\alpha + \beta + \gamma = 1$, we have
  \begin{align}
    \sum_{i=1}^n a_i^\alpha b_i^\beta c_i^\gamma
    &\leq
    \prn*{
      \sum_{i=1}^n a_i
    }^\alpha
    \prn*{
      \sum_{i=1}^n b_i
    }^\beta
    \prn*{
      \sum_{i=1}^n c_i
    }^\gamma.
    \label{eq:holder_3}
  \end{align}
\end{lemma}
Applying \cref{eq:holder_2} with $(\alpha, \beta) = (\frac{1}{p}, 1 - \frac{1}{p})$ gives
\begin{align}
  \sum_{i=0}^{k-1} \lambda_i \norm*{s_i}^{p-1}
  =
  \sum_{i=0}^{k-1} \lambda_i^\frac{1}{p} \prn*{ \lambda_i \norm*{s_i}^p }^{1 - \frac{1}{p}}
  \leq
  \prn*{
    \sum_{i=0}^{k-1} \lambda_i
  }^{\frac{1}{p}}
  \prn*{
    \sum_{i=0}^{k-1} \lambda_i \norm*{s_i}^p
  }^{1 - \frac{1}{p}}.
  \label{eq:holder_example}
\end{align}
This inequality and similar applications of H\"older's inequality will be used repeatedly in the analyses.

\subsection{Guiding example: Gradient descent}
\label{sec:gd}
We now demonstrate the complexity analysis for a higher-order-regularized variant of gradient descent.
The algorithm defines the step $s_k = x_{k+1} - x_k$ by
\begin{align}
  s_k
  =
  \argmin_{s \in \R^d}
  \set*{
    \inner*{\nabla f(x_k)}{s}
    + \frac{\lambda_k}{p} \norm*{s}^p
  },
  \label{eq:hor_gd}
\end{align}
where $p > 2$ and $\lambda_k > 0$.
Here, $p = 2$ would recover the standard gradient descent, while our recipe uses $p > 2$.
If $\nabla f(x_k) = \0$, then $s_k = \0$.
Otherwise, the unique solution to the subproblem~\cref{eq:hor_gd} is given by
\begin{align}
  s_k = - \prn*{\lambda_k \norm*{\nabla f(x_k)}^{p-2}}^{- \frac{1}{p-1}} \nabla f(x_k),
  \label{eq:hor_gd_explicit}
\end{align}
which coincides with the $\beta$-normalized gradient descent update~\citep{chen2023generalized}.
For the analysis, we impose the following standard assumption.
\begin{assumption}
  \label{asm:gd}
  There exists $L > 0$ such that $\norm*{\nabla f(x) - \nabla f(y)} \leq L \norm*{x - y}$ for all $x, y \in \R^d$.
\end{assumption}

To derive the complexity bound, we first bound the one-step progress $f(x_{i+1}) - f(x_i)$ and the gradient norm $\norm*{\nabla f(x_i)}$ in terms of $\lambda_i$ and $\norm*{s_i}$.
The proof is similar to the case $p = 2$.
\begin{lemma}
  \label{lem:gd_bounds_f_grad}
  Suppose that \cref{asm:gd} holds and let $p > 2$.
  Then, for all $i \in \N$, we have
  \begin{align}
    f(x_{i+1}) - f(x_i)
    &\leq
    \frac{L}{2} \norm*{s_i}^2 - \lambda_i \norm*{s_i}^p,
    \label{eq:lem_gd_decrease}\\
    \norm*{\nabla f(x_i)}
    &\leq
    \lambda_i \norm*{s_i}^{p-1}.
    \label{eq:lem_gd_gradnorm}
  \end{align}
\end{lemma}
\begin{proof}
  The first-order optimality condition for subproblem~\cref{eq:hor_gd} gives
  \begin{align}
    \nabla f(x_i)
    = - \lambda_i \norm*{s_i}^{p-2} s_i,
  \end{align}
  which proves \cref{eq:lem_gd_gradnorm}.
  Combining the standard descent lemma under \cref{asm:gd} with this equation yields
  \begin{align}
    f(x_{i+1}) - f(x_i)
    &\leq
    \inner*{\nabla f(x_i)}{s_i}
    + \frac{L}{2} \norm*{s_i}^2
    =
    - \lambda_i \norm*{s_i}^p
    + \frac{L}{2} \norm*{s_i}^2,
  \end{align}
  which proves \cref{eq:lem_gd_decrease}.
\end{proof}

Combining these bounds yields the following general upper bound on the gradient norm.
Recall that $\Delta$ is defined in \cref{eq:def_delta}.
\begin{lemma}
  \label{lem:gd_rate_general}
  Suppose that \cref{asm:gd} holds and let $p > 2$.
  Then, for all $k \geq 1$, we have
  \begin{align}
    \min_{0 \leq i < k} \norm*{\nabla f(x_i)}
    \leq
    \frac{1}{k}
    \prn*{
      \sum_{i=0}^{k-1} \lambda_i
    }^{\frac{1}{p}}
    \prn*{
      2 \Delta
      + L \sum_{i=0}^{k-1} \prn*{\frac{L}{\lambda_i}}^{\frac{2}{p-2}}
    }^{\frac{p-1}{p}}.
    \label{eq:gd_rate_general}
  \end{align}
\end{lemma}
\begin{proof}
  We first derive an upper bound for $\sum_{i=0}^{k-1} \lambda_i \norm*{s_i}^p$.
  Rearranging \cref{eq:lem_gd_decrease} and applying \cref{lem:power_diff_bound} with $t = \norm*{s_i}$ gives
  \begin{align}
    f(x_{i+1}) - f(x_i)
    + \frac{\lambda_i}{2} \norm*{s_i}^p
    \leq
    \frac{1}{2} \prn*{
      L \norm*{s_i}^2
      - \lambda_i \norm*{s_i}^p
    }
    \leq
    \frac{L}{2} \prn*{\frac{L}{\lambda_i}}^{\frac{2}{p-2}}.
    \label{eq:gd_obj_dec_bound}
  \end{align}
  Summing this bound over $0 \leq i < k$ and rearranging terms yields
  \begin{align}
    \sum_{i=0}^{k-1} \lambda_i \norm*{s_i}^p
    &\leq
    2 \Delta
    + L \sum_{i=0}^{k-1} \prn*{\frac{L}{\lambda_i}}^{\frac{2}{p-2}},
    \label{eq:gd_sum_lamsp}
  \end{align}
  where we have used $\sum_{i=0}^{k-1} \prn*{f(x_i) - f(x_{i+1})} = f(x_0) - f(x_k) \leq \Delta$.

  Next, summing \cref{eq:lem_gd_gradnorm} over $0 \leq i < k$ and applying \cref{eq:holder_example} gives
  \begin{align}
    \sum_{i=0}^{k-1} \norm*{\nabla f(x_i)}
    \leq
    \sum_{i=0}^{k-1} \lambda_i \norm*{s_i}^{p-1}
    \leq
    \prn*{
      \sum_{i=0}^{k-1} \lambda_i
    }^{\frac{1}{p}}
    \prn*{
      \sum_{i=0}^{k-1} \lambda_i \norm*{s_i}^p
    }^{1 - \frac{1}{p}}.
    \label{eq:gd_sum_grad}
  \end{align}
  Plugging \cref{eq:gd_sum_lamsp} into this bound and using $\min_{0 \leq i < k} \norm*{\nabla f(x_i)} \leq \frac{1}{k} \sum_{i=0}^{k-1} \norm*{\nabla f(x_i)}$ completes the proof.
\end{proof}

We emphasize that \cref{eq:gd_rate_general} holds for any positive sequence $(\lambda_k)_{k \in \N}$; this is the main benefit of choosing $p > 2$.
In the standard analysis for $p = 2$, one typically requires $\lambda_k > \frac{L}{2}$ to guarantee monotone decrease of the objective value, as follows from \cref{eq:lem_gd_decrease}.
With $p > 2$, monotone decrease is not necessarily guaranteed at every iteration, but the possible increase is controlled through higher-order regularization, as quantified in \cref{eq:gd_obj_dec_bound}.
Consequently, we obtain the bound \cref{eq:gd_rate_general} for arbitrary $\lambda_k > 0$.
This robustness to misspecification of $\lambda_k$ is the central mechanism behind our parameter-free construction.

The final step of the analysis is to turn the general bound \cref{eq:gd_rate_general} into explicit convergence rates by specifying $(\lambda_k)_{k \in \N}$.
Choosing $\lambda_k$ to balance the terms in \cref{eq:gd_rate_general} yields the following bounds.

\begin{theorem}
  \label{thm:gd_rate}
  Suppose that \cref{asm:gd} holds.
  Let $p > 2$ and $c_\lambda > 0$ be arbitrary constants.
  \begin{itemize}
    \item
    Set $\lambda_k = c_\lambda (k+1)^{\frac{p-2}{2}}$ for all $k \in \N$.
    Then, the following holds for all $k \geq 1$:
    \begin{align}
      \min_{0 \leq i < k} \norm*{\nabla f(x_i)}
      &\leq
      \frac{c_\lambda^{\frac{1}{p}}}{\sqrt{k}}
      \prn*{
        2 \Delta + L \prn*{\frac{L}{c_\lambda}}^{\frac{2}{p-2}} (1 + \log k)
      }^{\frac{p-1}{p}}
      =
      \tilde \O \prn*{k^{-1/2}}.
      \label{eq:gd_rate_log}
    \end{align}
    \item
    Fix an integer $K \geq 1$ and set $\lambda_k = c_\lambda K^{\frac{p-2}{2}}$ for all $0 \leq k < K$.
    Then, the following holds:
    \begin{align}
      \min_{0 \leq i < K} \norm*{\nabla f(x_i)}
      &\leq
      \frac{c_\lambda^{\frac{1}{p}}}{\sqrt{K}}
      \prn*{
        2 \Delta + L \prn*{\frac{L}{c_\lambda}}^{\frac{2}{p-2}}
      }^{\frac{p-1}{p}}
      =
      \O \prn*{K^{-1/2}}.
      \label{eq:gd_rate}
    \end{align}
  \end{itemize}
\end{theorem}
\begin{proof}
  If $\lambda_k = c_\lambda (k+1)^{\frac{p-2}{2}}$ for all $k \in \N$, then we can bound the sums in \cref{eq:gd_rate_general} as
  \begin{align}
    \sum_{i=0}^{k-1} \lambda_i
    \leq
    k \lambda_{k-1}
    =
    c_\lambda k^{\frac{p}{2}},\quad
    \sum_{i=0}^{k-1} \prn*{\frac{L}{\lambda_i}}^{\frac{2}{p-2}}
    &=
    \prn*{\frac{L}{c_\lambda}}^{\frac{2}{p-2}} \sum_{i=1}^k \frac{1}{i}
    \leq
    \prn*{\frac{L}{c_\lambda}}^{\frac{2}{p-2}} (1 + \log k).
  \end{align}
  Substituting these bounds into \cref{eq:gd_rate_general} gives the first result~\cref{eq:gd_rate_log}.
  The second result~\cref{eq:gd_rate} is obtained in a similar manner.
\end{proof}

The bound in \cref{eq:gd_rate_log} shows that the algorithm achieves $\tilde{\O}(\epsilon^{-2})$ complexity without prior knowledge of $L$, $\Delta$, or the total number of iterations $K$.
If $K$ is specified in advance, the bound~\cref{eq:gd_rate} gives $\O(\epsilon^{-2})$ complexity.
If $L$ and $\Delta$ are known up to constant factors, then setting $c_\lambda = \Theta \prn[\big]{\Delta^{\frac{2-p}{2}} L^{\frac{p}{2}}}$ in \cref{eq:gd_rate} yields the simplified rate
\begin{align}
  \min_{0 \leq i < K} \norm*{\nabla f(x_i)}
  &\leq
  \O \prn[\bigg]{ \sqrt{\frac{\Delta L}{K}} }.
\end{align}
The resulting complexity is $\O(\Delta L \epsilon^{-2})$, which is optimal in its dependence on $\Delta$ and $L$ as well as on $\epsilon$~\citep{carmon2020lower}.

Although the bound \cref{eq:gd_rate} assumes that $K$ is specified in advance, this assumption can be removed by using a doubling trick (e.g., \citep[Section~2.3.1]{shalev2012online}), as described in \cref{alg:doubling_trick}.
At stage $t = 0,1,\dots$, we set $K = 2^t$ and run $K$ iterations of~\cref{eq:hor_gd} from the same initial point $x_0$, using the constant regularization parameter $\lambda_k = c_\lambda K^{\frac{p-2}{2}}$ throughout that stage.
The restart from $x_0$ is needed to obtain the desired guarantee because the objective value is not guaranteed to decrease monotonically.
Applying \cref{eq:gd_rate} to each stage gives the following guarantee.

\begin{corollary}
  \label{cor:gd_rate_doubling_trick}
  Suppose that \cref{asm:gd} holds.
  Let $p > 2$ and $c_\lambda > 0$ be arbitrary constants.
  Let $(A_t)_{t \in \N}$ be the sequence of sets generated by \cref{alg:doubling_trick}.
  Then, the following holds for all $t \in \N$ and $K = 2^t$:
  \begin{align}
    \min_{x \in A_t} \norm*{\nabla f(x)}
    \leq
    \frac{c_\lambda^{\frac{1}{p}}}{\sqrt{K}}
    \prn*{
      2 \Delta + L \prn*{\frac{L}{c_\lambda}}^{\frac{2}{p-2}}
    }^{\frac{p-1}{p}}.
  \end{align}
\end{corollary}

\begin{algorithm}[t]
  \caption{Higher-order regularized gradient descent with the doubling trick}
  \label{alg:doubling_trick}
  \begin{algorithmic}[1]
    \Require{$x_0 \in \R^d$, $p > 2$, $c_\lambda > 0$}
    \For{$t = 0, 1, \dots$}
      \State $K \gets 2^t$
      \For{$k = 0, 1, \dots, K-1$}
        \State Compute $s_k$ by \cref{eq:hor_gd} with $\lambda_k = c_\lambda K^{\frac{p-2}{2}}$
        \State $x_{k+1} \gets x_k + s_k$
      \EndFor
      \State $A_t \gets \set*{x_0, x_1, \dots, x_K}$
    \EndFor
  \end{algorithmic}
\end{algorithm}

The set $A_t$ is obtained after $\sum_{i=0}^t 2^i = 2^{t+1} - 1 = 2 K - 1$ inner iterations, where $K = 2^t$.
Thus, the doubling trick increases the total number of iterations by less than a factor of two and preserves the same complexity up to a universal constant factor.
The same argument applies to the other algorithms considered in the subsequent sections, so we present only their guarantees for a fixed number of iterations.

\paragraph{Related work.}
The higher-order regularized gradient descent~\cref{eq:hor_gd} is closely related to normalized gradient descent (NGD).
Earlier normalized methods have been studied in quasiconvex optimization \citep{kiwiel2001convergence,hazan2015beyond}.
Most closely related to our instantiation is the $\beta$-NGD of \citet{chen2023generalized}, which uses
\begin{align}
  x_{k+1}
  =
  x_k - \frac{\eta_k}{\norm*{\nabla f(x_k)}^\beta} \nabla f(x_k),
  \label{eq:beta_ngd}
\end{align}
where $\beta \in [0,1]$ and $\eta_k > 0$.
With suitable parameter choices, $\beta$-NGD achieves the optimal $\O(\epsilon^{-2})$ complexity under generalized smoothness assumptions.
For the case $\beta = 1$, \citet{yang2023two} obtained a parameter-free $\tilde \O(\epsilon^{-2})$ bound.
Our update~\cref{eq:hor_gd_explicit} is equivalent to $\beta$-NGD with $\beta = \frac{p-2}{p-1}$, and our analysis gives parameter-free $\O(\epsilon^{-2})$ bounds.
The main point is that this update arises from a higher-order-regularized subproblem, which then serves as a template for more general model-based algorithms.

\subsection{The general recipe for model-based methods}
\label{sec:recipe_general}
The gradient-descent example above illustrates the main mechanism of the recipe.
We now extract a basic template for analyzing model-based methods of the form \cref{eq:subproblem_hor}.
This template captures the core argument used throughout the paper, although some of the algorithms analyzed later require minor modifications to individual steps.

\paragraph{Step 1: Bound $f(x_{i+1}) - f(x_i)$ and $\norm*{\nabla f(x_i)}$ by $\lambda_i$ and $\norm*{s_i}$.}
For each $i \in \N$, we first derive upper bounds on $f(x_{i+1}) - f(x_i)$ and $\norm*{\nabla f(x_i)}$ in terms of $\lambda_i$ and $\norm*{s_i}$.
In the gradient-descent example, this step is carried out in \cref{lem:gd_bounds_f_grad}.
Such bounds are usually obtained using two ingredients: an upper bound on the error of the local model, typically derived from Lipschitz continuity of the relevant derivative, and the first-order optimality condition for the subproblem.
This part of the argument is close to the analysis of the corresponding method with the standard choice of regularization.

\paragraph{Step 2: Bound $\sum_{i=0}^{k-1} \lambda_i \norm*{s_i}^p$ by $\lambda_0,\dots,\lambda_{k-1}$.}
We next convert the upper bound on $f(x_{i+1}) - f(x_i)$ obtained in Step~1 into an upper bound on $\sum_{i=0}^{k-1} \lambda_i \norm*{s_i}^p$.
The bound from Step~1 contains the negative regularization term $-\lambda_i \norm*{s_i}^p$, together with lower-order positive terms in $\norm*{s_i}$ arising from the model error.
The key operation is to retain a positive fraction of $\lambda_i \norm*{s_i}^p$ on the left-hand side and use the remaining fraction to control the right-hand side independently of $\norm*{s_i}$.
This is the point at which the higher-order regularization is essential.
Since the regularization exponent $p$ is larger than the order of the model error, \cref{lem:power_diff_bound} can be applied with $t = \norm*{s_i}$ to eliminate $\norm*{s_i}$ from the right-hand side, as in \cref{eq:gd_obj_dec_bound}.
Summing the resulting inequalities over $0 \leq i < k$ then gives the desired bound, as in \cref{eq:gd_sum_lamsp}.

\paragraph{Step 3: Bound $\sum_{i=0}^{k-1} \norm*{\nabla f(x_i)}$ by $\sum_{i=0}^{k-1} \lambda_i \norm*{s_i}^p$.}
Summing the gradient-norm bound obtained in Step~1 over $0 \leq i < k$, we obtain an upper bound on $\sum_{i=0}^{k-1} \norm*{\nabla f(x_i)}$.
This upper bound typically involves mixed terms containing both $\lambda_i$ and $\norm*{s_i}$.
The key operation is to apply H\"older's inequality to bound these mixed terms in terms of $\sum_{i=0}^{k-1} \lambda_i \norm*{s_i}^p$ and sums depending only on the $\lambda_i$'s, as in \cref{eq:gd_sum_grad}.

\paragraph{Step 4: Obtain convergence rates.}
Combining the bounds from Steps~2 and 3 yields a general upper bound on $\sum_{i=0}^{k-1} \norm*{\nabla f(x_i)}$ in terms of $\lambda_0,\dots,\lambda_{k-1}$, as in \cref{lem:gd_rate_general}.
Specifying $(\lambda_k)_{k \in \N}$ to balance the terms in this bound then gives explicit convergence rates, as in \cref{thm:gd_rate}.

\bigskip

In the following sections, this mechanism is instantiated in different forms for Newton's method, the Gauss--Newton method, SGD, and PAGE.

\section{Newton's method}
\label{sec:newton}
This section applies the higher-order regularization recipe to Newton's method.

\subsection{Algorithm and assumptions}
\label{sec:newton_algorithm}
At each iteration $k \in \N$, the algorithm computes the step $s_k \coloneqq x_{k+1} - x_k$ by approximately solving the following subproblem:
\begin{align}
  \min_{s \in \R^d}\ 
  \set*{
    m_k(s)
    \coloneqq
    \inner*{g_k}{s}
    + \frac{1}{2} \inner*{H_k s}{s}
    + \frac{\lambda_k}{p} \norm*{s}^p
  },
  \label{eq:newton_subproblem}
\end{align}
where $g_k \coloneqq \nabla f(x_k)$, $H_k \coloneqq \nabla^2 f(x_k)$, $p > 3$, and $\lambda_k > 0$.
This subproblem can be solved by standard methods for regularized quadratic subproblems.
For example, one may use methods based on secular equations~\citep{gould2010solving} or Krylov subspace methods~\citep{gould2020error}.

The algorithm considered in this section allows for inexact solutions to the subproblem~\cref{eq:newton_subproblem}.
More specifically, we assume that $s_k$ satisfies the following conditions for all $k \in \N$:
\begin{subequations}
\begin{align}
    \norm*{\nabla m_k(s_k)}
    &\leq
    \frac{\lambda_k}{2} \norm*{s_k}^{p-1},
    \label{eq:newton_approx_stationarity}\\
    \inner*{g_k}{s_k}
    &\leq
    0.
    \label{eq:newton_approx_descent}
\end{align}
\end{subequations}
The first condition requires $s_k$ to satisfy an approximate first-order optimality condition for the subproblem, which is a common requirement in regularized Newton methods; see, e.g.,~\citep{cartis2011adaptive,cartis2011adaptive2,cartis2019universal}.
The second condition is a mild descent-type requirement, and it is equivalent to $m_k(s_k) \leq m_k(-s_k)$.
These conditions are automatically satisfied when $s_k$ is a global minimizer of $m_k$.

We use the following standard assumption in the subsequent complexity analysis; see, e.g., \citep{nesterov2006cubic,cartis2011adaptive,cartis2011adaptive2}.
\begin{assumption}
  \label{asm:newton}
  There exists $M > 0$ such that $\norm*{\nabla^2 f(x) - \nabla^2 f(y)} \leq M \norm*{x - y}$ for all $x, y \in \R^d$.
\end{assumption}
Under this assumption, the following Taylor-type bounds hold for every iteration $k \in \N$ (e.g., \citep[Lemma~1.2.4]{nesterov2018lectures}):
\begin{align}
  f(x_{k+1}) - f(x_k)
  &\leq
  \inner*{g_k}{s_k} + \frac{1}{2} \inner*{H_k s_k}{s_k} + \frac{M}{6} \norm*{s_k}^3,
  \label{eq:newton_taylor_obj}\\
  \norm*{g_{k+1} - g_k - H_k s_k}
  &\leq
  \frac{M}{2} \norm*{s_k}^2.
  \label{eq:newton_taylor_grad}
\end{align}

\subsection{Complexity analysis}
As the first step of the recipe, we derive upper bounds on $f(x_{i+1}) - f(x_i)$ and $\norm*{\nabla f(x_{i+1})}$ in terms of $\lambda_i$ and $\norm*{s_i}$.
Here we bound $\norm*{\nabla f(x_{i+1})}$ rather than $\norm*{\nabla f(x_i)}$; this is standard for Newton-type methods (see, e.g., \citep{nesterov2006cubic,cartis2011adaptive}).
\begin{lemma}
  Suppose that \cref{asm:newton} holds and let $p > 3$.
  Then, for all $i \in \N$, we have
  \begin{align}
    f(x_{i+1}) - f(x_i)
    &\leq
    \frac{M}{6} \norm*{s_i}^3
    - \frac{\lambda_i}{4} \norm*{s_i}^p,
    \label{eq:lem_newton_decrease}\\
    \norm*{\nabla f(x_{i+1})}
    &\leq
    \frac{M}{2} \norm*{s_i}^2
    + \frac{3}{2} \lambda_i \norm*{s_i}^{p-1}.
    \label{eq:lem_newton_gradnorm}
  \end{align}
\end{lemma}
\begin{proof}
  Since $\nabla m_i(s) = g_i + H_i s + \lambda_i \norm*{s}^{p-2} s$, condition~\cref{eq:newton_approx_stationarity} can be rewritten as
  \begin{align}
    \norm*{g_i + H_i s_i + \lambda_i \norm*{s_i}^{p-2} s_i}
    \leq
    \frac{\lambda_i}{2} \norm*{s_i}^{p-1}.
    \label{eq:newton_approx_stationarity_rewrite}
  \end{align}
  We have
  \begin{alignat}{2}
    f(x_{i+1}) - f(x_i)
    - \frac{M}{6} \norm*{s_i}^3
    &\leq
    \frac{1}{2} \inner*{g_i + H_i s_i}{s_i}
    &&\by{\cref{eq:newton_approx_descent,eq:newton_taylor_obj}}\\
    &=
    \frac{1}{2} \inner*{g_i + H_i s_i + \lambda_i \norm*{s_i}^{p-2} s_i}{s_i}
    - \frac{\lambda_i}{2} \norm*{s_i}^p\\
    &\leq
    \frac{\lambda_i}{4} \norm*{s_i}^p
    - \frac{\lambda_i}{2} \norm*{s_i}^p,
    &&\by{\cref{eq:newton_approx_stationarity_rewrite}}
  \end{alignat}
  which proves \cref{eq:lem_newton_decrease}.
  Using the triangle inequality, we have
  \begin{alignat}{2}
    \norm*{g_{i+1}}
    &\leq
    \norm*{
      g_{i+1} - g_i - H_i s_i
    }
    + \norm*{g_i + H_i s_i + \lambda_i \norm*{s_i}^{p-2} s_i}
    + \lambda_i \norm*{s_i}^{p-1}\\
    &\leq
    \frac{M}{2} \norm*{s_i}^2
    + \frac{\lambda_i}{2} \norm*{s_i}^{p-1}
    + \lambda_i \norm*{s_i}^{p-1},
    &&\by{\cref{eq:newton_taylor_grad,eq:newton_approx_stationarity_rewrite}}
  \end{alignat}
  which proves \cref{eq:lem_newton_gradnorm}.
\end{proof}

Using the above lemma, we derive a general upper bound on $\min_{1 \leq i \leq k} \norm*{\nabla f(x_i)}$ in terms of $\lambda_0,\dots,\lambda_{k-1}$.
The argument follows Steps 2 and 3 of the recipe in \cref{sec:recipe_general}.
\begin{lemma}
  \label{lem:newton_rate_general}
  Suppose that \cref{asm:newton} holds and let $p > 3$.
  Then, for all $k \geq 1$, we have
  \begin{align}
    \min_{1 \leq i \leq k} \norm*{\nabla f(x_i)}
    &\leq
    \frac{2}{k} \prn*{
      \sum_{i=0}^{k-1} \lambda_i
    }^{\frac{1}{p}}
    \prn*{
      12 \Delta + 2 M \sum_{i=0}^{k-1} \prn*{\frac{M}{\lambda_i}}^{\frac{3}{p-3}}
    }^{1 - \frac{1}{p}}.
    \label{eq:newton_rate_general}
  \end{align}
\end{lemma}
\begin{proof}
  We first derive an upper bound for $\sum_{i=0}^{k-1} \lambda_i \norm*{s_i}^p$.
  Rearranging \cref{eq:lem_newton_decrease} and applying \cref{lem:power_diff_bound} gives
  \begin{align}
    f(x_{i+1}) - f(x_i)
    + \frac{\lambda_i}{12} \norm*{s_i}^p
    \leq
    \frac{1}{6} \prn*{
      M \norm*{s_i}^3
      - \lambda_i \norm*{s_i}^p
    }
    \leq
    \frac{M}{6} \prn*{\frac{M}{\lambda_i}}^{\frac{3}{p-3}}.
  \end{align}
  Summing this bound over $0 \leq i < k$ and rearranging terms yields
  \begin{align}
    \sum_{i=0}^{k-1} \lambda_i \norm*{s_i}^p
    &\leq
    12 \Delta
    + 2 M \sum_{i=0}^{k-1} \prn*{\frac{M}{\lambda_i}}^{\frac{3}{p-3}}
    \eqqcolon A_k,
    \label{eq:newton_sum_lamsp}
  \end{align}
  where we have used $\sum_{i=0}^{k-1} \prn*{f(x_i) - f(x_{i+1})} = f(x_0) - f(x_k) \leq \Delta$.

  Next, summing \cref{eq:lem_newton_gradnorm} over $0 \leq i < k$ yields
  \begin{align}
    \sum_{i=1}^k \norm*{\nabla f(x_i)}
    \leq
    \frac{M}{2} \sum_{i=0}^{k-1} \norm*{s_i}^2
    + \frac{3}{2} \sum_{i=0}^{k-1} \lambda_i \norm*{s_i}^{p-1}.
    \label{eq:newton_sum_gradnorm_proof1}
  \end{align}
  The second sum on the right-hand side is bounded as in \cref{eq:holder_example}:
  \begin{align}
    \sum_{i=0}^{k-1} \lambda_i \norm*{s_i}^{p-1}
    &\leq
    \prn*{
      \sum_{i=0}^{k-1} \lambda_i
    }^{\frac{1}{p}}
    \prn*{
      \sum_{i=0}^{k-1} \lambda_i \norm*{s_i}^p
    }^{1 - \frac{1}{p}}
    \leq
    \prn*{
      \sum_{i=0}^{k-1} \lambda_i
    }^{\frac{1}{p}}
    A_k^{1 - \frac{1}{p}}.
  \end{align}
  Similarly, the first sum on the right-hand side of \cref{eq:newton_sum_gradnorm_proof1} is bounded by H\"older's inequality~\cref{eq:holder_3} with $(\alpha, \beta, \gamma) = (\frac{1}{p}, \frac{2}{p}, 1 - \frac{3}{p})$ as follows:
  \begin{align}
    M \sum_{i=0}^{k-1} \norm*{s_i}^2
    &\leq
    \prn*{
      \sum_{i=0}^{k-1} \lambda_i
    }^{\!\frac{1}{p}}
    \!
    \prn*{
      \sum_{i=0}^{k-1} \lambda_i \norm*{s_i}^p
    }^{\!\frac{2}{p}}
    \!
    \prn*{
      M \sum_{i=0}^{k-1} \prn*{\frac{M}{\lambda_i}}^{\frac{3}{p-3}}
    }^{\!\!1 - \frac{3}{p}}
    \leq
    \prn*{
      \sum_{i=0}^{k-1} \lambda_i
    }^{\!\frac{1}{p}}
    A_k^{1 - \frac{1}{p}},
  \end{align}
  where the second inequality uses \cref{eq:newton_sum_lamsp} and the bound $M \sum_{i=0}^{k-1} \prn[\big]{\frac{M}{\lambda_i}}^{\frac{3}{p-3}} \leq A_k$, which follows from the definition of $A_k$ in \cref{eq:newton_sum_lamsp}.
  Plugging these bounds into \cref{eq:newton_sum_gradnorm_proof1} gives
  \begin{align}
    \sum_{i=1}^k \norm*{\nabla f(x_i)}
    &\leq
    2 \prn*{
      \sum_{i=0}^{k-1} \lambda_i
    }^{\frac{1}{p}}
    A_k^{1 - \frac{1}{p}}.
  \end{align}
  Using $\min_{1 \leq i \leq k} \norm*{\nabla f(x_i)} \leq \frac{1}{k} \sum_{i=1}^k \norm*{\nabla f(x_i)}$ completes the proof.
\end{proof}

Now we specify $(\lambda_k)_{k \in \N}$ to obtain explicit convergence rates.
\begin{theorem}
  \label{thm:newton_rate}
  Suppose that \cref{asm:newton} holds.
  Let $p > 3$ and $c_\lambda > 0$ be arbitrary constants.
  Fix an integer $K \geq 1$ and set $\lambda_k = c_\lambda K^{\frac{p-3}{3}}$ for all $0 \leq k < K$.
  Then, the following holds:
  \begin{align}
    \min_{1 \leq i \leq K} \norm*{\nabla f(x_i)}
    &\leq
    \frac{2 c_\lambda^{\frac{1}{p}}}{K^{2/3}}
    \prn[\bigg]{
      12 \Delta
      + 2 M \prn*{\frac{M}{c_\lambda}}^{\frac{3}{p-3}}
    }^{1 - \frac{1}{p}}
    =
    \O \prn*{K^{-2/3}}.
    \label{eq:newton_rate}
  \end{align}
\end{theorem}
\begin{proof}
  When $\lambda_k = c_\lambda K^{\frac{p-3}{3}}$, the sums in \cref{eq:newton_rate_general} are evaluated as
  \begin{align}
    \sum_{i=0}^{K-1} \lambda_i
    =
    c_\lambda K^{\frac{p}{3}},\quad
    \sum_{i=0}^{K-1} \prn*{\frac{M}{\lambda_i}}^{\frac{3}{p-3}}
    &=
    \prn*{\frac{M}{c_\lambda}}^{\frac{3}{p-3}}.
  \end{align}
  Substituting these equations into \cref{eq:newton_rate_general} completes the proof.
\end{proof}

The bound \cref{eq:newton_rate} provides the complexity bound $\O(\epsilon^{-3/2})$.
If $M$ and $\Delta$ are known, setting $c_\lambda = \Theta \prn[\big]{\Delta^{\frac{3-p}{3}} M^{\frac{p}{3}}}$ in \cref{eq:newton_rate} yields the following bound:
\begin{align}
  \min_{1 \leq i \leq K} \norm*{\nabla f(x_i)}
  &\leq
  \O \prn[\bigg]{ \frac{\Delta^{2/3} M^{1/3}}{K^{2/3}} }.
\end{align}
The resulting complexity is $\O(\Delta \sqrt{M} \epsilon^{-3/2})$, matching the lower bound~\citep{carmon2020lower} in its dependence on $\Delta$, $M$, and $\epsilon$.

\paragraph{Related work.}
The seminal cubic-regularized Newton method of \citet{nesterov2006cubic} achieves the complexity $\O(\epsilon^{-3/2})$ for functions with Lipschitz-continuous Hessians.
This line of work has been extended in several directions~\citep{cartis2011adaptive,cartis2011adaptive2,curtis2019inexact,birgin2017worst,grapiglia2017regularized,cartis2019universal,curtis2023worst,gratton2023convergence}.
\citet{grapiglia2017regularized} study regularized Newton methods with regularization exponents $p \in [2, 3]$ for H\"older-continuous Hessians.
\citet{cartis2019universal} show that the $\O(\epsilon^{-3/2})$ complexity can be achieved with $p > 3$, but their framework uses trial-point acceptance tests.
\citet{gratton2023convergence} establish the same complexity bound for the case $p = 3$, using adaptive regularization parameters without acceptance tests.
Our analysis establishes a bound for arbitrary regularization parameters $(\lambda_k)_{k \in \N}$ and exponents $p > 3$, as in \cref{lem:newton_rate_general}.
This result provides a parameter-free guarantee with regularization parameters fixed in advance.

\section{Gauss--Newton method}
\label{sec:gn}
This section applies the higher-order regularization recipe to the Gauss--Newton method.
The Gauss--Newton method is designed for nonlinear least-squares problems of the form
\begin{align}
  \min_{x \in \R^d}\ 
  \set*{
    f(x)
    \coloneqq
    \frac{1}{2} \norm*{F(x)}^2
  },
  \label{eq:gn_problem}
\end{align}
where $F \colon \R^d \to \R^n$ is a differentiable function.
Let $J(x) \in \R^{n \times d}$ be the Jacobian matrix of $F$ at $x$.

\subsection{Algorithm and assumptions}
We compute the step $s_k \coloneqq x_{k+1} - x_k$ by approximately solving the following subproblem based on the Gauss--Newton approximation $\norm*{F(x_k+s)}^2 \simeq \norm*{F_k + J_k s}^2$:
\begin{align}
  \min_{s \in \R^d}\ 
  \set*{
    m_k(s)
    \coloneqq
    \frac{1}{2} \norm*{F_k + J_k s}^2
    + \frac{\lambda_k}{p} \norm*{s}^p
  },
  \label{eq:gn_subproblem}
\end{align}
where $F_k \coloneqq F(x_k)$, $J_k \coloneqq J(x_k)$, $p > 2$, and $\lambda_k > 0$.
Nesterov's accelerated gradient method~\citep{nesterov1983method} is a natural choice for solving this subproblem.
For this subproblem, the method achieves an $\O(t^{-2p/(p-2)})$ convergence rate~\citep{nesterov2022inexact,roulet2017sharpness}, where $t$ is the number of iterations.
This rate is known to be optimal~\citep{thomsen2024complexity}.
The algorithm in this section also allows for inexact solutions to the subproblem~\cref{eq:gn_subproblem}.
Unlike in the previous section, we assume only that $s_k$ satisfies the approximate first-order optimality condition for the subproblem:
\begin{align}
  \norm*{\nabla m_k(s_k)}
  \leq
  \frac{\lambda_k}{2} \norm*{s_k}^{p-1}.
  \label{eq:gn_subproblem_inexactness}
\end{align}

For the standard choice $p = 2$, the method based on subproblem~\cref{eq:gn_subproblem} is often referred to as the Levenberg--Marquardt method~\citep{levenberg1944method,marquardt1963algorithm}, and several choices of $\lambda_k$ have been studied.
One common choice is to set $\lambda_k$ proportional to $\norm*{F_k}$ \citep{marumo2023majorization,fan2003modified}.
Following this convention, we write
\begin{align}
  \lambda_k = 2 \mu_k \norm*{F_k}
  \label{eq:gn_lambda_mu}
\end{align}
with $\mu_k > 0$ throughout this section.
This is only a change of notation, but it is useful because $\mu_k$ will play the same role here as $\lambda_k$ does in the general recipe of \cref{sec:recipe}.

For the analysis, we use the following standard assumption.
\begin{assumption}
  \label{asm:gn}
  Let $L, \sigma > 0$ be constants.
  \begin{enuminasm}
    \item \label{asm:gn_lip_jac}
    $\norm*{J(x) - J(y)} \leq L \norm*{x - y}$ for all $x, y \in \R^d$.
    \item \label{asm:gn_bdd_jac}
    $\norm*{J(x)} \leq \sigma$ for all $x \in \R^d$.
  \end{enuminasm}
\end{assumption}
Under \cref{asm:gn_lip_jac}, the following standard inequality holds:
\begin{align}
  \norm*{F_{k+1} - F_k - J_k s_k}
  \leq
  \frac{L}{2} \norm*{s_k}^2.
  \label{eq:gn_F_taylor_bound}
\end{align}
We may assume without loss of generality that $\norm*{F_k} > 0$ for all $k \in \N$.
Indeed, if $\norm*{F_k} = 0$ for some $k$, then $x_k$ is a global minimizer of $f$, and the algorithm can be terminated.

\subsection{Complexity analysis}
As the first step of the recipe, we derive upper bounds on $\norm*{F_{i+1}} - \norm*{F_i}$ and $\norm*{\nabla f(x_{i+1})}$ in terms of $\mu_i$ and $\norm*{s_i}$.
Here we work with $\norm*{F_i}$ rather than $f(x_i) = \frac{1}{2} \norm*{F_i}^2$ because \cref{eq:gn_F_taylor_bound} gives a bound for the residual $F$ itself.
\begin{lemma}
  Suppose that \cref{asm:gn} holds and let $p > 2$.
  Then, for all $i \in \N$, we have
  \begin{align}
    \norm*{F_{i+1}} - \norm*{F_i}
    &\leq
    \frac{L}{2} \norm*{s_i}^2 - \mu_i \norm*{s_i}^p,
    \label{eq:gn_F_upperbound}\\
    \norm*{\nabla f(x_{i+1})}
    &\leq
    \norm*{F_i} \prn*{
      L \norm*{s_i}
      + 3 \mu_i \norm*{s_i}^{p-1}
    }
    + \frac{L\sigma}{2} \norm*{s_i}^2.
    \label{eq:gn_grad_upperbound}
  \end{align}
\end{lemma}
The proof of this lemma is inspired by existing analyses of the Levenberg--Marquardt method \citep[Lemmas~3.1 and 4.1]{marumo2025accelerated}.
\begin{proof}
  Since $\nabla m_i(s) = J_i^\top (F_i + J_i s) + \lambda_i \norm*{s}^{p-2} s$, condition~\cref{eq:gn_subproblem_inexactness} can be rewritten as
  \begin{align}
    \norm[\big]{J_i^\top (F_i + J_i s_i) + \lambda_i \norm*{s_i}^{p-2} s_i}
    \leq
    \frac{\lambda_i}{2} \norm*{s_i}^{p-1}.
    \label{eq:gn_subproblem_inexactness_rewrite}
  \end{align}
  We have
  \begin{alignat}{2}
    \norm*{F_i + J_i s_i}^2 - \norm*{F_i}^2
    &\leq
    \norm*{F_i + J_i s_i}^2 - \norm*{F_i}^2 + \norm*{J_i s_i}^2\\
    &=
    2 \inner*{J_i^\top (F_i + J_i s_i) + \lambda_i \norm*{s_i}^{p-2} s_i}{s_i}
    - 2 \lambda_i \norm*{s_i}^p\\
    &\leq
    \lambda_i \norm*{s_i}^p
    - 2 \lambda_i \norm*{s_i}^p
    =
    - \lambda_i \norm*{s_i}^p
    =
    - 2 \mu_i \norm*{F_i} \norm*{s_i}^p,
    \label{eq:gn_model_decrease}
  \end{alignat}
  where the second inequality follows from Cauchy--Schwarz and \cref{eq:gn_subproblem_inexactness_rewrite}.
  Rearranging this inequality yields
  \begin{align}
    \norm*{F_i + J_i s_i}
    \leq
    \sqrt{ \norm*{F_i}^2 - 2 \mu_i \norm*{F_i} \norm*{s_i}^p }
    =
    \norm*{F_i} \sqrt{ 1 - 2 \mu_i \frac{\norm*{s_i}^p}{\norm*{F_i}} }
    \leq
    \norm*{F_i} - \mu_i \norm*{s_i}^p,
    \label{eq:gn_F_Js_upperbound}
  \end{align}
  where the last inequality holds because $\sqrt{1 - 2t} \leq 1 - t$ whenever $1 - 2t \geq 0$.
  Using the triangle inequality and \cref{eq:gn_F_taylor_bound} gives
  \begin{align}
    \norm*{F_{i+1}}
    &\leq
    \norm*{F_i + J_i s_i}
    + \norm*{F_{i+1} - F_i - J_i s_i}
    \leq
    \norm*{F_i + J_i s_i}
    + \frac{L}{2} \norm*{s_i}^2.
    \label{eq:gn_Fk1_upperbound}
  \end{align}
  Plugging \cref{eq:gn_F_Js_upperbound} into this bound gives the first result~\cref{eq:gn_F_upperbound}.

  For the second result, we decompose $\norm*{\nabla f(x_{i+1})}$ as follows:
  \begin{align}
    \norm*{\nabla f(x_{i+1})}
    &=
    \norm[\big]{J_{i+1}^\top F_{i+1}}\\
    &=
    \norm[\big]{
      J_{i+1}^\top \prn*{F_{i+1} - F_i - J_i s_i}
      + \prn*{J_{i+1} - J_i}^\top \prn*{F_i + J_i s_i}
      + J_i^\top (F_i + J_i s_i)
    }\\
    &\leq
    \norm*{J_{i+1}} \norm*{F_{i+1} - F_i - J_i s_i}
    + \norm*{J_{i+1} - J_i} \norm*{F_i + J_i s_i}
    + \norm[\big]{J_i^\top (F_i + J_i s_i)}.
  \end{align}
  We bound the first four norms using \cref{asm:gn_bdd_jac}, \cref{eq:gn_F_taylor_bound}, \cref{asm:gn_lip_jac}, and \cref{eq:gn_F_Js_upperbound}, respectively:
  \begin{align}
    \norm*{\nabla f(x_{i+1})}
    &\leq
    \sigma \cdot \prn*{\frac{L}{2} \norm*{s_i}^2}
    + \prn*{L \norm*{s_i}} \cdot \prn*{\norm*{F_i} - \mu_i \norm*{s_i}^p}
    + \norm[\big]{J_i^\top (F_i + J_i s_i)}\\
    &\leq
    \frac{L\sigma}{2} \norm*{s_i}^2
    + L \norm*{F_i} \norm*{s_i}
    + \norm[\big]{J_i^\top (F_i + J_i s_i)}.
  \end{align}
  The last term is bounded using \cref{eq:gn_subproblem_inexactness_rewrite} as follows:
  \begin{align}
    \norm[\big]{J_i^\top (F_i + J_i s_i)}
    \leq
    \norm[\big]{J_i^\top (F_i + J_i s_i) + \lambda_i \norm*{s_i}^{p-2} s_i}
    + \lambda_i \norm*{s_i}^{p-1}
    &\leq
    \frac{3}{2} \lambda_i \norm*{s_i}^{p-1}\\
    &=
    3 \mu_i \norm*{F_i} \norm*{s_i}^{p-1},
  \end{align}
  where the last equality uses \cref{eq:gn_lambda_mu}.
  Plugging this bound into the previous inequality completes the proof of \cref{eq:gn_grad_upperbound}.
\end{proof}

Using the above lemma, we derive a general upper bound on $\min_{1 \leq i \leq k} \norm*{\nabla f(x_i)}$ in terms of $\mu_0,\dots,\mu_{k-1}$.
The argument largely follows Steps~2 and 3 of the recipe in \cref{sec:recipe_general}.
One difference is that \cref{eq:gn_grad_upperbound} contains $\norm*{F_i}$ in addition to $\mu_i$ and $\norm*{s_i}$, so we also need to bound $\norm*{F_i}$.
For this reason, in Step~2 we bound $2\norm*{F_k} + \sum_{i=0}^{k-1} \mu_i \norm*{s_i}^p$ rather than only $\sum_{i=0}^{k-1} \mu_i \norm*{s_i}^p$.

\begin{lemma}
  \label{lem:gn_rate_general}
  Suppose that \cref{asm:gn} holds and let $p > 2$.
  Then, for all $k \geq 1$, we have
  \begin{align}
    \min_{1 \leq i \leq k} \norm*{\nabla f(x_i)}
    &\leq
    \frac{2 A_k^{2 - \frac{1}{p}}}{k} \prn*{\sum_{i=0}^{k-1} \mu_i}^{\frac{1}{p}}
    + \frac{\sigma A_k}{2 k},
  \end{align}
  where
  \begin{align}
    A_k
    \coloneqq
    2 \norm*{F_0} + L \sum_{i=0}^{k-1} \prn*{\frac{L}{\mu_i}}^{\frac{2}{p-2}}.
    \label{eq:gn_def_Sk}
  \end{align}
\end{lemma}
\begin{proof}
  We first derive an upper bound for $2 \norm*{F_k} + \sum_{i=0}^{k-1} \mu_i \norm*{s_i}^p$.
  Rearranging \cref{eq:gn_F_upperbound} and applying \cref{lem:power_diff_bound} gives
  \begin{align}
    \norm*{F_{i+1}} - \norm*{F_i}
    + \frac{\mu_i}{2} \norm*{s_i}^p
    &\leq
    \frac{1}{2} \prn*{
      L \norm*{s_i}^2 - \mu_i \norm*{s_i}^p
    }
    \leq
    \frac{L}{2} \prn*{\frac{L}{\mu_i}}^{\frac{2}{p-2}}.
  \end{align}
  Summing this bound over $0 \leq i < k$ and rearranging terms yields
  \begin{align}
    2 \norm*{F_k}
    + \sum_{i=0}^{k-1} \mu_i \norm*{s_i}^p
    &\leq
    2 \norm*{F_0}
    + L \sum_{i=0}^{k-1} \prn*{\frac{L}{\mu_i}}^{\frac{2}{p-2}}
    =
    A_k.
    \label{eq:gn_sum_Fk_sp_bound}
  \end{align}

  Next, summing \cref{eq:gn_grad_upperbound} over $0 \leq i < k$ gives
  \begin{align}
    \sum_{i=1}^k \norm*{\nabla f(x_i)}
    &\leq
    \prn*{\max_{0 \leq i < k} \norm*{F_i}}
    \prn*{
      L \sum_{i=0}^{k-1} \norm*{s_i}
      + 3 \sum_{i=0}^{k-1} \mu_i \norm*{s_i}^{p-1}
    }
    + \frac{L\sigma}{2} \sum_{i=0}^{k-1} \norm*{s_i}^2\\
    &\leq
    \frac{A_k}{2}
    \prn*{
      L \sum_{i=0}^{k-1} \norm*{s_i}
      + 3 \sum_{i=0}^{k-1} \mu_i \norm*{s_i}^{p-1}
    }
    + \frac{L\sigma}{2} \sum_{i=0}^{k-1} \norm*{s_i}^2,
    \label{eq:gn_sum_gradnorm_proof1}
  \end{align}
  where the second inequality uses $\norm*{F_i} \leq \frac{A_i}{2} \leq \frac{A_k}{2}$ from \cref{eq:gn_sum_Fk_sp_bound}.
  As in the proofs of \cref{lem:gd_rate_general,lem:newton_rate_general}, we bound the three sums using H\"older's inequality as follows:
  \begin{align}
    L &\sum_{i=0}^{k-1} \norm*{s_i}
    \leq
    \prn*{\sum_{i=0}^{k-1} \mu_i}^{\!\frac{1}{p}}
    \prn*{\sum_{i=0}^{k-1} \mu_i \norm*{s_i}^p}^{\!\frac{1}{p}}
    \prn*{\sum_{i=0}^{k-1} L \prn*{\frac{L}{\mu_i}}^{\frac{2}{p-2}}}^{\!\!1 - \frac{2}{p}}
    \leq
    \prn*{\sum_{i=0}^{k-1} \mu_i}^{\!\frac{1}{p}}
    A_k^{1 - \frac{1}{p}},\\
    &\sum_{i=0}^{k-1} \mu_i \norm*{s_i}^{p-1}
    \leq
    \prn*{\sum_{i=0}^{k-1} \mu_i}^{\!\frac{1}{p}}
    \prn*{\sum_{i=0}^{k-1} \mu_i \norm*{s_i}^p}^{\!\!1 - \frac{1}{p}}
    \leq
    \prn*{\sum_{i=0}^{k-1} \mu_i}^{\!\frac{1}{p}}
    A_k^{1 - \frac{1}{p}},\\
    L &\sum_{i=0}^{k-1} \norm*{s_i}^2
    \leq
    \prn*{\sum_{i=0}^{k-1} \mu_i \norm*{s_i}^p}^{\!\frac{2}{p}}
    \prn*{L \sum_{i=0}^{k-1} \prn*{\frac{L}{\mu_i}}^{\frac{2}{p-2}}}^{\!\!1 - \frac{2}{p}}
    \leq
    A_k,
  \end{align}
  where the second inequality in each line is obtained by applying the bounds $\sum_{i=0}^{k-1} \mu_i \norm*{s_i}^p \leq A_k$ and $L \sum_{i=0}^{k-1} \prn[\big]{\frac{L}{\mu_i}}^{\frac{2}{p-2}} \leq A_k$, which follow from \cref{eq:gn_sum_Fk_sp_bound}.
  Plugging these bounds into \cref{eq:gn_sum_gradnorm_proof1} yields
  \begin{align}
    \sum_{i=1}^k \norm*{\nabla f(x_i)}
    &\leq
    2 A_k^{2 - \frac{1}{p}}
    \prn*{\sum_{i=0}^{k-1} \mu_i}^{\!\frac{1}{p}}
    + \frac{\sigma}{2} A_k.
  \end{align}
  Using $\min_{1 \leq i \leq k} \norm*{\nabla f(x_i)} \leq \frac{1}{k} \sum_{i=1}^k \norm*{\nabla f(x_i)}$ completes the proof.
\end{proof}

Now we specify $(\mu_k)_{k \in \N}$ to obtain explicit convergence rates.
\begin{theorem}
  \label{thm:gn_rate}
  Suppose that \cref{asm:gn} holds.
  Let $p > 2$ and $c_\mu > 0$ be arbitrary constants.
  Fix an integer $K \geq 1$ and set $\mu_k = c_\mu K^{\frac{p}{2} - 1}$ for all $0 \leq k < K$.
  Then, the following holds:
  \begin{align}
    \min_{1 \leq i \leq K} \norm*{\nabla f(x_i)}
    &\leq
    \frac{2 c_\mu^{\frac{1}{p}} C^{2 - \frac{1}{p}}}{\sqrt{K}}
    + \frac{\sigma C}{2 K}
    =
    \O \prn*{K^{-1/2}},
    \quad\text{where}\quad
    C
    \coloneqq
    2 \norm*{F_0} + L \prn*{\frac{L}{c_\mu}}^{\frac{2}{p-2}}.
    \label{eq:gn_rate}
  \end{align}
\end{theorem}
The proof follows by substituting the specified choice of $(\mu_k)_{k \in \N}$ into \cref{lem:gn_rate_general}, and is omitted.

The bound \cref{eq:gn_rate} provides the complexity bound $\O(\epsilon^{-2})$.
If $L$ is known, setting $c_\mu = \Theta \prn[\big]{\norm*{F_0}^{\frac{2-p}{2}} L^{\frac{p}{2}}}$ in \cref{eq:gn_rate} yields $C = \O(\norm*{F_0})$ and the following bound:
\begin{align}
  \min_{1 \leq i \leq K} \norm*{\nabla f(x_i)}
  &\leq
  \O \prn[\bigg]{ \frac{\norm*{F_0}^{3/2} \sqrt{L}}{\sqrt{K}} + \frac{\sigma \norm*{F_0}}{K} }.
\end{align}
This rate yields the complexity bound
\begin{align}
  \O \prn[\bigg]{
    \frac{L \norm*{F_0}^3}{\epsilon^2}
    + \frac{\sigma \norm*{F_0}}{\epsilon}
  },
\end{align}
which recovers the state-of-the-art bound~\citep{marumo2025accelerated} for least-squares problems of the form~\cref{eq:gn_problem}.
To the best of our knowledge, the optimality of this complexity bound remains open.

\paragraph{Related work.}
Gauss--Newton (GN) methods are designed for nonlinear least-squares problems of the form \cref{eq:gn_problem}.
For global complexity guarantees, they are typically combined with regularization or trust-region mechanisms.
\citet{ueda2010global} established an $\O(\epsilon^{-2})$ complexity bound for quadratically regularized GN methods under \cref{asm:gn}; related $\O(\epsilon^{-2})$ results were later obtained in \citep{zhao2016global,bellavia2018levenberg,bergou2020convergence,marumo2023majorization,marumo2025accelerated}.
Among them, \citet{marumo2025accelerated} made the dependence on problem parameters explicit and showed the advantage of GN methods over gradient descent.
Most of the above methods are parameter-free but rely on acceptance tests.
A different line of work has studied cubic regularization for GN methods, mainly for local convergence~\citep{bellavia2015strong}.
Our instantiation requires no acceptance test and achieves the same state-of-the-art complexity as~\citep{marumo2025accelerated}.

\section{SGD with mini-batches}
\label{sec:sgd}
This section applies the higher-order regularization recipe to stochastic gradient descent (SGD) with mini-batches.
For notational simplicity, we focus on the finite-sum setting:
\begin{align}
  \min_{x \in \R^d}\
  \set[\bigg]{
    f(x)
    \coloneqq
    \frac{1}{n} \sum_{i=1}^n f_i(x)
  },
  \label{eq:sgd_problem}
\end{align}
where each $f_i \colon \R^d \to \R$ is a differentiable function.
We also write $f(x) = \ex<i>{f_i(x)}$, where $\ex<i>{\cdot}$ denotes expectation with respect to the uniform distribution on $\set{1,\dots,n}$.
The arguments in this section extend directly to the expectation setting $f(x) = \ex<\xi \sim P>{f(x;\xi)}$ for a general distribution $P$, provided that an unbiased gradient estimator is available.

\subsection{Algorithm and assumptions}
We set the step $s_k \coloneqq x_{k+1} - x_k$ by
\begin{align}
  s_k
  =
  \argmin_{s \in \R^d}
  \set*{
    \inner*{g_k}{s}
    + \frac{\lambda_k}{p} \norm*{s}^p
  }
  =
  - \prn*{ \lambda_k \norm*{g_k}^{p-2} }^{-\frac{1}{p-1}} g_k,
  \label{eq:sgd_subproblem}
\end{align}
where $p > 2$, $\lambda_k > 0$, and
\begin{align}
  g_k
  \coloneqq
  \frac{1}{B} \sum_{i \in I_k} \nabla f_i(x_k)
  \label{eq:sgd_grad_estimator}
\end{align}
is the mini-batch gradient estimator.
Here, $I_k$ is a uniformly sampled subset of $\set{1,\dots,n}$ of size $B$.
Throughout the stochastic sections, whenever a prescribed mini-batch size exceeds $n$, we use the full batch and interpret the corresponding mini-batch gradient estimator as the full gradient.

For the analysis, we use the following standard assumptions.
\begin{assumption}
  \label{asm:sgd}
  Let $L, \sigma > 0$ be constants.
  \begin{enuminasm}
    \item
    \label{asm:sgd_lip_grad}
    $\norm*{\nabla f(x) - \nabla f(y)} \leq L \norm*{x - y}$ for all $x, y \in \R^d$.
    \item
    \label{asm:sgd_bounded_variance}
    $\ex[\big]<i>{\norm*{\nabla f_i(x) - \nabla f(x)}^2} \leq \sigma^2$ for all $x \in \R^d$.
  \end{enuminasm}
\end{assumption}
Let
\begin{align}
  e_k \coloneqq \norm*{g_k - \nabla f(x_k)}
  \label{eq:sgd_def_ek}
\end{align}
be the error of the gradient estimator.
Under \cref{asm:sgd_bounded_variance}, the definition of $g_k$ in \cref{eq:sgd_grad_estimator} gives
\begin{align}
  \ex*{e_k^2}
  \leq
  \frac{\sigma^2}{B}.
  \label{eq:sgd_error_variance}
\end{align}

\subsection{Complexity analysis}
As the first step, we derive upper bounds on $f(x_{i+1}) - f(x_i)$ and $\norm*{\nabla f(x_i)}$ in terms of $\lambda_i$, $\norm*{s_i}$, and $e_i$.
The argument is similar to \cref{lem:gd_bounds_f_grad}, but we need to account for the error $e_i$.
\begin{lemma}
  \label{lem:sgd_bounds_f_grad}
  Suppose that \cref{asm:sgd_lip_grad} holds and let $p > 2$.
  Then, for all $i \in \N$, we have
  \begin{align}
    f(x_{i+1}) - f(x_i)
    &\leq
    L \norm*{s_i}^2
    - \lambda_i \norm*{s_i}^p
    + \frac{e_i^2}{2 L},
    \label{eq:sgd_function_decrease}\\
    \norm*{\nabla f(x_i)}
    &\leq
    \lambda_i \norm*{s_i}^{p-1} + e_i.
    \label{eq:sgd_gradient_upperbound}
  \end{align}
\end{lemma}
\begin{proof}
  The first-order optimality condition of subproblem~\cref{eq:sgd_subproblem} gives
  \begin{align}
    g_i = - \lambda_i \norm*{s_i}^{p-2} s_i.
    \label{eq:sgd_subproblem_optimality}
  \end{align}
  Using the triangle inequality and this equation yields
  \begin{align}
    \norm*{\nabla f(x_i)}
    \leq
    \norm*{g_i} + e_i
    =
    \lambda_i \norm*{s_i}^{p-1} + e_i,
  \end{align}
  which proves \cref{eq:sgd_gradient_upperbound}.
  We have
  \begin{alignat}{2}
    f(x_{i+1}) - f(x_i)
    - \frac{L}{2} \norm*{s_i}^2
    &\leq
    \inner*{\nabla f(x_i)}{s_i}
    &\quad&\by{\cref{asm:sgd_lip_grad}}\\
    &=
    \inner*{g_i}{s_i} - \inner*{g_i - \nabla f(x_i)}{s_i}\\
    &\leq
    \inner*{g_i}{s_i} + e_i \norm*{s_i}
    &\quad&\by{Cauchy--Schwarz}\\
    &=
    - \lambda_i \norm*{s_i}^p + e_i \norm*{s_i}.
    &\quad&\by{\cref{eq:sgd_subproblem_optimality}}
  \end{alignat}
  Applying Young's inequality $e_i \norm*{s_i} \leq \frac{e_i^2}{2 L} + \frac{L}{2} \norm*{s_i}^2$ completes the proof of \cref{eq:sgd_function_decrease}.
\end{proof}

For $k \geq 1$, let $\tilde x_k$ be chosen uniformly at random from $x_0,\dots,x_{k-1}$, independently of all other randomness.
Using the above lemma, we derive a general upper bound on $\ex*{\norm*{\nabla f(\tilde x_k)}}$ in terms of $(\lambda_i)_{i=0}^{k-1}$.
The argument largely follows Steps~2 and 3 of the recipe in \cref{sec:recipe_general}.
Compared with the deterministic case, the analysis requires two additional ingredients.
First, in Step~2, we bound $\sum_{i=0}^{k-1} \ex*{\lambda_i \norm*{s_i}^p + e_i^2 / L}$ rather than only $\sum_{i=0}^{k-1} \ex*{\lambda_i \norm*{s_i}^p}$.
The additional term $e_i^2 / L$ will be used to control the error term $e_i$ in \cref{eq:sgd_gradient_upperbound}.
Second, in Step~3, after taking expectations in \cref{eq:sgd_gradient_upperbound}, we use Jensen's inequality to bound $\ex*{\norm*{\nabla f(x_i)}}$ in terms of $\ex*{\norm*{s_i}^p}$ and $\ex*{e_i^2}$.
\begin{lemma}
  \label{lem:sgd_rate_general}
  Suppose that \cref{asm:sgd} holds and let $p > 2$.
  Then, for all $k \geq 1$, we have
  \begin{align}
    \ex*{\norm*{\nabla f(\tilde x_k)}}
    &\leq
    \frac{2}{k}
    \prn*{
      \sum_{i=0}^{k-1} \lambda_i
    }^{\frac{1}{p}}
    \prn*{
      2 \Delta
      + 2 L \sum_{i=0}^{k-1} \prn*{\frac{2 L}{\lambda_i}}^{\frac{2}{p-2}}
      + \frac{2 k \sigma^2}{L B}
    }^{\frac{p-1}{p}}.
    \label{eq:sgd_rate_general}
  \end{align}
\end{lemma}
\begin{proof}
  We first bound the sum $\sum_{i=0}^{k-1} \ex*{\lambda_i \norm*{s_i}^p + e_i^2/L}$.
  Rearranging \cref{eq:sgd_function_decrease} and applying \cref{lem:power_diff_bound} gives
  \begin{align}
    f(x_{i+1}) - f(x_i)
    + \frac{\lambda_i}{2} \norm*{s_i}^p
    - \frac{e_i^2}{2 L}
    &\leq
    L \norm*{s_i}^2
    - \frac{\lambda_i}{2} \norm*{s_i}^p
    \leq
    L \prn*{\frac{2 L}{\lambda_i}}^{\frac{2}{p-2}}.
  \end{align}
  Summing this bound over $0 \leq i < k$ and rearranging terms yields
  \begin{align}
    \sum_{i=0}^{k-1} \prn*{ \lambda_i \norm*{s_i}^p + \frac{e_i^2}{L} }
    &\leq
    2 \Delta
    + 2 L \sum_{i=0}^{k-1} \prn*{\frac{2 L}{\lambda_i}}^{\frac{2}{p-2}}
    + \sum_{i=0}^{k-1} \frac{2 e_i^2}{L},
  \end{align}
  where we have used $f(x_0) - f(x_k) \leq \Delta$.
  Taking expectation and applying \cref{eq:sgd_error_variance} to the last term yields
  \begin{align}
    \sum_{i=0}^{k-1} \ex*{ \lambda_i \norm*{s_i}^p + \frac{e_i^2}{L} }
    &\leq
    2 \Delta
    + 2 L \sum_{i=0}^{k-1} \prn*{\frac{2 L}{\lambda_i}}^{\frac{2}{p-2}}
    + \frac{2 k \sigma^2}{L B}
    \eqqcolon
    A_k,
    \label{eq:sgd_sum_s_bound}
  \end{align}

  Next, taking expectation in \cref{eq:sgd_gradient_upperbound} and applying Jensen's inequality, we have
  \begin{align}
    \ex*{\norm*{\nabla f(x_i)}}
    &\leq
    \lambda_i \ex*{\norm*{s_i}^{p-1}} + \ex*{e_i}
    \leq
    \lambda_i \ex*{\norm*{s_i}^p}^{\frac{p-1}{p}} + \sqrt{\ex*{e_i^2}}.
  \end{align}
  Summing the above bound over $0 \leq i < k$ gives
  \begin{align}
    k
    \ex*{\norm*{\nabla f(\tilde x_k)}}
    =
    \sum_{i=0}^{k-1} \ex*{\norm*{\nabla f(x_i)}}
    &\leq
    \sum_{i=0}^{k-1} \lambda_i \ex*{\norm*{s_i}^p}^{\frac{p-1}{p}}
    + \sum_{i=0}^{k-1} \sqrt{\ex*{e_i^2}}.
  \end{align}
  As in \cref{lem:newton_rate_general,lem:gn_rate_general}, we bound the two sums using H\"older's inequality and \cref{eq:sgd_sum_s_bound}:
  \begin{align}
    &\sum_{i=0}^{k-1} \lambda_i \ex*{\norm*{s_i}^p}^{\frac{p-1}{p}}
    \leq
    \prn*{\sum_{i=0}^{k-1} \lambda_i}^{\frac{1}{p}}
    \prn*{\sum_{i=0}^{k-1} \lambda_i \ex*{\norm*{s_i}^p} }^{\frac{p-1}{p}}
    \leq
    \prn*{\sum_{i=0}^{k-1} \lambda_i}^{\frac{1}{p}}
    A_k^{1 - \frac{1}{p}},\\
    &\sum_{i=0}^{k-1} \sqrt{\ex*{e_i^2}}
    \leq
    \prn*{\sum_{i=0}^{k-1} \lambda_i}^{\frac{1}{p}}
    \prn*{\sum_{i=0}^{k-1} \frac{\ex*{e_i^2}}{L}}^{\frac{1}{2}}
    \prn*{L \sum_{i=0}^{k-1} \prn*{\frac{L}{\lambda_i}}^{\frac{2}{p-2}}}^{\frac{p-2}{2p}}
    \leq
    \prn*{\sum_{i=0}^{k-1} \lambda_i}^{\frac{1}{p}}
    A_k^{1 - \frac{1}{p}}.
  \end{align}
  Plugging these bounds into the previous inequality and dividing by $k$ completes the proof.
\end{proof}

Now we specify $(\lambda_k)_{k \in \N}$ and $B$ to obtain explicit convergence rates.
\begin{theorem}
  \label{thm:sgd_rate}
  Suppose that \cref{asm:sgd} holds.
  Let $p > 2$ and $c_\lambda, c_B > 0$ be arbitrary constants.
  Fix an integer $K \geq 1$ and set $\lambda_k = c_\lambda K^{\frac{p-2}{2}}$ for all $0 \leq k < K$ and $B = \ceil*{c_B K}$.
  Then, the following holds:
  \begin{align}
    \ex*{\norm*{\nabla f(\tilde x_K)}}
    &\leq
    \frac{2 c_\lambda^{\frac{1}{p}}}{\sqrt{K}}
    \prn*{
      2 \Delta
      + 2 L \prn*{\frac{2 L}{c_\lambda}}^{\frac{2}{p-2}}
      + \frac{2 \sigma^2}{L c_B}
    }^{\frac{p-1}{p}}
    =
    \O \prn*{ K^{-1/2} }.
    \label{eq:sgd_rate}
  \end{align}
  Furthermore, the oracle complexity to achieve $\ex*{\norm*{\nabla f(\tilde x_K)}} \leq \epsilon$ is $\O \prn[\big]{\epsilon^{-4}}$.
\end{theorem}
\begin{proof}
  We omit the proof of \cref{eq:sgd_rate}, as it is almost the same as the proof of \cref{thm:gd_rate}.
  The convergence rate in \cref{eq:sgd_rate} implies that $K = \O \prn*{\epsilon^{-2}}$ suffices to ensure $\ex*{\norm*{\nabla f(\tilde x_K)}} \leq \epsilon$.
  Hence, the total number of evaluations of $\nabla f_i$ is
  \begin{align}
    K B
    =
    \O \prn*{K \prn*{1 + c_B K}}
    =
    \O \prn*{ \epsilon^{-4} },
  \end{align}
  which completes the proof.
\end{proof}

If $L$, $\Delta$, and $\sigma$ are known, setting $c_\lambda = \Theta \prn[\big]{\Delta^{\frac{2-p}{2}} L^{\frac{p}{2}}}$ and $c_B = \Theta \prn[\big]{ \frac{\sigma^2}{\Delta L} }$ in \cref{eq:sgd_rate} yields
$
  \ex*{\norm*{\nabla f(\tilde x_K)}}
  \leq
  \O \prn[\big]{
    \sqrt{\Delta L / K}
  }
$.
Then, the resulting iteration complexity is $K = \O \prn*{\Delta L \epsilon^{-2}}$, and the oracle complexity is
\begin{align}
  K B
  =
  \O \prn*{K \prn*{1 + c_B K}}
  =
  \O \prn*{
    \frac{\Delta L}{\epsilon^2}
    \prn*{
      1 + \frac{\sigma^2}{\epsilon^2}
    }
  },
\end{align}
which matches the optimal dependence on $\Delta$, $L$, $\sigma$, and $\epsilon$~\citep{carmon2020lower,arjevani2023lower}.

\paragraph{Related work.}
Under \cref{asm:sgd}, \citet{ghadimi2013stochastic} established the standard SGD complexity $\O(\Delta L \epsilon^{-2} (1 + \sigma^2 \epsilon^{-2}))$, which is known to be optimal~\citep{carmon2020lower,arjevani2023lower}.
Untuned SGD achieves a $\tilde \O(\epsilon^{-4})$ bound, but its dependence on $L$ can be exponential~\citep{yang2023two}.
Parameter-free methods that avoid this exponential dependence include AdaGrad-Norm and normalized SGD.
For AdaGrad-Norm, an $\O(\epsilon^{-4})$ bound was proved under an additional bounded-gradient assumption~\citep{ward2020adagrad}, and this assumption was later removed at the cost of logarithmic factors~\citep{faw2022power}.
For normalized SGD, a $\tilde \O(\epsilon^{-4})$ bound was obtained with momentum~\citep{yang2023two}, while a log-free $\O(\epsilon^{-4})$ bound was established with mini-batching~\citep{hubler2025gradient}.
The bound in~\citep{hubler2025gradient} also recovers the optimal dependence $\O(\Delta L \epsilon^{-2} (1 + \sigma^2 \epsilon^{-2}))$ when the parameters are tuned.
Our instantiation is closest to~\citep{hubler2025gradient}.
Because our analysis is based on a model-based framework, it also extends naturally to variance-reduced methods.

\section{The PAGE algorithm}
\label{sec:page}
This section applies the higher-order regularization recipe to PAGE, a variance-reduced stochastic gradient method proposed by \citet{li2021page}.
We continue to consider the finite-sum problem \cref{eq:sgd_problem}.

\subsection{Algorithm}
We use the same update rule as in \cref{sec:sgd}; that is, the step $s_k \coloneqq x_{k+1} - x_k$ is computed by solving \cref{eq:sgd_subproblem}.
The difference is that $g_k$ is the PAGE estimator~\citep{li2021page}, defined by \cref{eq:sgd_grad_estimator} for $k=0$ and by
\begin{align}
  g_k
  =
  \begin{dcases}
    \displaystyle
    \frac{1}{B} \sum_{i \in I_k} \nabla f_i(x_k),
    & \text{with probability } \theta,\\
    \displaystyle
    g_{k-1}
    +
    \frac{1}{b} \sum_{i \in J_k}
    \prn*{\nabla f_i(x_k) - \nabla f_i(x_{k-1})},
    & \text{with probability } 1 - \theta
  \end{dcases}
  \label{eq:page_estimator}
\end{align}
for $k \geq 1$.
Here, $\theta \in (0, 1]$, and $I_k$ and $J_k$ are uniformly sampled subsets of $\set{1,\dots,n}$ of sizes $B$ and $b$, respectively.

The PAGE analysis uses the following standard assumptions.
\begin{assumption}
  \label{asm:page}
  Let $L, \sigma > 0$ be constants.
  \begin{enuminasm}
    \item
    \label{asm:page_average_smoothness}
    $\ex[\big]<i>{\norm*{\nabla f_i(x) - \nabla f_i(y)}^2} \leq L^2 \norm*{x - y}^2$ for all $x, y \in \R^d$.
    \item
    \label{asm:page_bounded_variance}
    $\ex[\big]<i>{\norm*{\nabla f_i(x) - \nabla f(x)}^2} \leq \sigma^2$ for all $x \in \R^d$.
  \end{enuminasm}
\end{assumption}
\Cref{asm:page_average_smoothness} is often referred to as average smoothness; it is stronger than the smoothness condition on $f$.
Indeed, by Jensen's inequality, \cref{asm:page_average_smoothness} implies \cref{asm:sgd_lip_grad}.

As in the SGD analysis, we define the estimation error $e_k$ by \cref{eq:sgd_def_ek}.
This error satisfies the following recursion.
\begin{lemma}[{\citep[Lemmas~3 and 4]{li2021page}}]
  \label{lem:page_error_recursion}
  Suppose that \cref{asm:page_average_smoothness} holds.
  \begin{itemize}
    \item
    If $B = n$, then the following holds for all $i \in \N$:
    \begin{align}
      \ex*{e_{i+1}^2 - (1 - \theta) e_i^2}
      \leq
      L^2 \frac{1 - \theta}{b} \ex[\big]{ \norm*{s_i}^2 }.
      \label{eq:page_error_recursion_exact}
    \end{align}
    \item
    If \cref{asm:page_bounded_variance} holds, then the following holds for all $i \in \N$:
    \begin{align}
      \ex*{e_{i+1}^2 - (1 - \theta) e_i^2}
      \leq
      L^2 \frac{1 - \theta}{b} \ex[\big]{ \norm*{s_i}^2 }
      + \frac{\sigma^2 \theta}{B}.
      \label{eq:page_error_recursion}
    \end{align}
  \end{itemize}
\end{lemma}
The proof in \citep{li2021page} relies only on the update rule for $g_k$ in \cref{eq:page_estimator}, and is independent of the update rule for $x_k$.
Hence, the lemma above also applies to our setting.

PAGE is known to achieve two oracle complexity bounds:
$\O(n + \sqrt{n}\epsilon^{-2})$ in the exact-refresh case where $B = n$, and
$\O(\epsilon^{-3})$ in the bounded-variance setting (\cref{asm:page_bounded_variance}).
In the following analysis, we derive parameter-free counterparts of both bounds.

For notational simplicity, we state both results for the finite-sum problem~\cref{eq:sgd_problem}.
As in the SGD section, the bounded-variance analysis extends directly to the general expectation setting $f(x) = \ex<\xi \sim P>{f(x;\xi)}$.
The exact-refresh case, however, requires the finite-sum structure, because it sets $B = n$ and uses the full gradient at refresh iterations.

\subsection{Complexity analysis}
Since \cref{asm:page_average_smoothness} implies \cref{asm:sgd_lip_grad}, and PAGE differs from SGD only in the construction of $g_k$, \cref{lem:sgd_bounds_f_grad} applies to PAGE as well.
This completes the first step of the recipe in \cref{sec:recipe_general}.

As in the SGD analysis, for $k \geq 1$, let $\tilde x_k$ be chosen uniformly at random from $x_0,\dots,x_{k-1}$, independently of all other randomness.
The following lemma provides a general bound on $\ex*{\norm*{\nabla f(\tilde x_k)}}$ in the bounded-variance setting.
The proof is inspired by \citep[Section~3]{li2021short}.

\begin{lemma}
  \label{lem:page_rate_general}
  Suppose that \cref{asm:page} holds and let $p > 2$.
  Set
  \begin{align}
    \theta
    =
    \frac{1}{1 + b}.
    \label{eq:page_theta_b_relation}
  \end{align}
  Then, for all $k \geq 1$, we have
  \begin{align}
    \ex*{\norm*{\nabla f(\tilde x_k)}}
    \leq
    \frac{2}{k}
    \prn*{
      \sum_{i=0}^{k-1} \lambda_i
    }^{\frac{1}{p}}
    \prn*{
      2 \Delta
      + 4 L \sum_{i=0}^{k-1} \prn*{\frac{4 L}{\lambda_i}}^{\frac{2}{p-2}}
      + \frac{2 \sigma^2 (k + b)}{L B}
    }^{\frac{p-1}{p}}.
    \label{eq:page_rate_general}
  \end{align}
\end{lemma}
\begin{proof}
  We first bound the sum $\sum_{i=0}^{k-1} \ex*{e_i^2}$.
  Since $(1 + b)(1 - \theta) = b$ and $(1 + b) \theta = 1$ by \cref{eq:page_theta_b_relation}, multiplying \cref{eq:page_error_recursion} by $(1 + b)$ gives
  \begin{align}
    \ex*{ (1 + b) e_{i+1}^2 - b e_i^2 }
    &\leq
    L^2 \ex[\big]{ \norm*{s_i}^2 }
    + \frac{\sigma^2}{B}.
  \end{align}
  Summing this bound over $0 \leq i \leq k-2$ yields
  \begin{align}
    (1 + b) \ex*{ e_{k-1}^2}
    + \sum_{i=1}^{k-2} \ex*{ e_i^2}
    - b \ex*{e_0^2}
    &\leq
    L^2 \sum_{i=0}^{k-2} \ex[\big]{ \norm*{s_i}^2 }
    + \frac{\sigma^2 (k-1)}{B}.
  \end{align}
  Rearranging terms gives
  \begin{align}
    \sum_{i=0}^{k-1} \ex*{ e_i^2}
    &\leq
    L^2 \sum_{i=0}^{k-2} \ex[\big]{ \norm*{s_i}^2 }
    + \frac{\sigma^2 (k-1)}{B}
    + (1+b) \ex*{e_0^2} - b \ex*{e_{k-1}^2}\\
    &\leq
    L^2 \sum_{i=0}^{k-1} \ex[\big]{ \norm*{s_i}^2 }
    + \frac{\sigma^2 (k + b)}{B},
    \label{eq:page_error_sum_bound}
  \end{align}
  where the last inequality uses $\ex*{e_0^2} \leq \sigma^2 / B$.

  Next, rearranging \cref{eq:sgd_function_decrease} and applying \cref{lem:power_diff_bound} gives
  \begin{align}
    f(x_{i+1}) - f(x_i)
    + \frac{\lambda_i}{2} \norm*{s_i}^p
    + L \norm*{s_i}^2
    - \frac{e_i^2}{2 L}
    &\leq
    2 L \norm*{s_i}^2
    - \frac{\lambda_i}{2} \norm*{s_i}^p
    \leq
    2 L \prn*{\frac{4 L}{\lambda_i}}^{\frac{2}{p-2}}.
  \end{align}
  Summing this bound over $0 \leq i < k$ and rearranging terms gives
  \begin{align}
    \sum_{i=0}^{k-1} \prn*{ \lambda_i \norm*{s_i}^p + 2 L \norm*{s_i}^2 - \frac{e_i^2}{L} }
    &\leq
    2 \Delta
    + 4 L \sum_{i=0}^{k-1} \prn*{\frac{4 L}{\lambda_i}}^{\frac{2}{p-2}}.
  \end{align}
  Taking expectation and adding \cref{eq:page_error_sum_bound} multiplied by $2/L$ gives
  \begin{align}
    \sum_{i=0}^{k-1} \ex*{ \lambda_i \norm*{s_i}^p + \frac{e_i^2}{L} }
    &\leq
    2 \Delta
    + 4 L \sum_{i=0}^{k-1} \prn*{\frac{4 L}{\lambda_i}}^{\frac{2}{p-2}}
    + \frac{2 \sigma^2 (k + b)}{L B}
    \eqqcolon
    A_k.
  \end{align}

  The remainder follows exactly as in the second paragraph of the proof of \cref{lem:sgd_rate_general}, with $A_k$ defined as above, and is omitted.
\end{proof}

The exact-refresh case follows from the same argument.
Indeed, in the proof of \cref{lem:page_rate_general}, if we use \cref{eq:page_error_recursion_exact} instead of \cref{eq:page_error_recursion}, then the term involving $\sigma$ disappears.
Consequently, under the exact-refresh setting $B = n$, we obtain
\begin{align}
  \ex*{\norm*{\nabla f(\tilde x_k)}}
  \leq
  \frac{2}{k}
  \prn*{
    \sum_{i=0}^{k-1} \lambda_i
  }^{\frac{1}{p}}
  \prn*{
    2 \Delta
    + 4 L \sum_{i=0}^{k-1} \prn*{\frac{4 L}{\lambda_i}}^{\frac{2}{p-2}}
  }^{\frac{p-1}{p}}.
  \label{eq:page_rate_general_finite-sum}
\end{align}
The proof is identical to that of \cref{lem:page_rate_general}, with \cref{eq:page_error_recursion_exact} in place of \cref{eq:page_error_recursion}, and is omitted.

The following theorem specifies the algorithmic parameters and gives the resulting complexity bound in the exact-refresh setting.
\begin{theorem}
  Suppose that \cref{asm:page_average_smoothness} holds.
  Let $p > 2$ and $c_\lambda > 0$ be arbitrary constants.
  Fix an integer $K \geq 1$ and set
  \begin{align}
    b = \ceil*{\sqrt{n}},\quad
    B = n,\quad
    \theta = \frac{1}{1 + b},
    \label{eq:page_params_finite-sum}
  \end{align}
  and $\lambda_k = c_\lambda K^{\frac{p-2}{2}}$ for all $0 \leq k < K$.
  Then, the following holds:
  \begin{align}
    \ex*{\norm*{\nabla f(\tilde x_K)}}
    &\leq
    \frac{2 c_\lambda^{\frac{1}{p}}}{\sqrt{K}}
    \prn*{
      2 \Delta
      + 4 L \prn*{\frac{4 L}{c_\lambda}}^{\frac{2}{p-2}}
    }^{\frac{p-1}{p}}
    =
    \O \prn*{ K^{-1/2} }.
    \label{eq:page_rate_finite-sum}
  \end{align}
  Furthermore, the oracle complexity to achieve $\ex*{\norm*{\nabla f(\tilde x_K)}} \leq \epsilon$ is $\O \prn[\big]{n + \sqrt{n} \epsilon^{-2}}$.
\end{theorem}
\begin{proof}
  The two sums in \cref{eq:page_rate_general_finite-sum} can be bounded in the same way as in the proof of \cref{thm:gd_rate}, yielding \cref{eq:page_rate_finite-sum}.
  It suffices to take $K = \O \prn*{\epsilon^{-2}}$ to ensure $\ex*{\norm*{\nabla f(\tilde x_K)}} \leq \epsilon$.
  Hence, the expected number of evaluations of $\nabla f_i$ is
  \begin{align}
    B + (K-1) \prn[\big]{
      \theta B + 2 (1 - \theta) b
    }
    &=
    n + \O \prn*{K \sqrt{n}}
    =
    \O \prn*{ n + \sqrt{n} \epsilon^{-2} },
  \end{align}
  which completes the proof.
\end{proof}

If $L$ and $\Delta$ are known, setting $c_\lambda = \Theta \prn[\big]{\Delta^{\frac{2-p}{2}} L^{\frac{p}{2}}}$ in \cref{eq:page_rate_finite-sum} yields
$
  \ex*{\norm*{\nabla f(\tilde x_K)}}
  \leq
  \O \prn[\big]{
    \sqrt{\Delta L / K}
  }
$.
Thus, the resulting iteration complexity is $K = \O \prn*{\Delta L \epsilon^{-2}}$, and the oracle complexity is $\O(n + \sqrt{n} \Delta L \epsilon^{-2})$.
This matches the lower bound~\citep{li2021page}.

We next turn to the bounded-variance setting.
Compared with the exact-refresh case, the parameter choice is more delicate because the additional variance term in \cref{eq:page_rate_general} must also be controlled.
The following theorem gives a parameter-free choice that leads to the oracle complexity $\O(\epsilon^{-3})$.
\begin{theorem}
  Suppose that \cref{asm:page} holds.
  Let $p > 2$ and $c_\lambda, c_b > 0$ be arbitrary constants.
  Fix an integer $K \geq 1$ and let $m \coloneqq \min \set[\big]{ \sqrt{K},\, K / c_b}$.
  Set
  \begin{align}
    b = \ceil*{c_b m},\quad
    B = b^2,\quad
    \theta = \frac{1}{1 + b},
  \end{align}
  and $\lambda_k = c_\lambda m^{p-2}$ for all $0 \leq k < K$.
  Then, the following holds:
  \begin{align}
    \ex*{\norm*{\nabla f(\tilde x_K)}}
    \leq
    \frac{2 c_\lambda^{\frac{1}{p}}}{m}
    \prn*{
      2 \Delta
      + 4 L \prn*{\frac{4 L}{c_\lambda}}^{\frac{2}{p-2}}
      + \frac{4 \sigma^2}{L c_b^2}
    }^{\frac{p-1}{p}}
    =
    \O \prn*{K^{-1/2}}.
    \label{eq:page_rate_bounded-variance}
  \end{align}
  Furthermore, the oracle complexity to achieve $\ex*{\norm*{\nabla f(\tilde x_K)}} \leq \epsilon$ is $\O \prn[\big]{\epsilon^{-3}}$.
\end{theorem}
\begin{proof}
  The sums involving $\lambda_i$ in \cref{eq:page_rate_general} can be evaluated as follows:
  \begin{align}
    \sum_{i=0}^{K-1} \lambda_i
    =
    c_\lambda m^{p-2} K,\quad
    \sum_{i=0}^{K-1} \prn*{\frac{4 L}{\lambda_i}}^{\frac{2}{p-2}}
    =
    \frac{K}{m^2}
    \prn*{\frac{4 L}{c_\lambda}}^{\frac{2}{p-2}}.
  \end{align}
  The term involving $\sigma$ can be bounded as follows:
  \begin{align}
    \frac{2 \sigma^2 (K + b)}{L B}
    =
    \frac{2 \sigma^2 (K + b)}{L b^2}
    \leq
    \frac{4 \sigma^2 K}{L c_b^2 m^2},
  \end{align}
  where the last inequality follows from $b \leq K$ and $b \geq c_b m$.
  Plugging the above bounds into \cref{eq:page_rate_general} and using $\frac{K}{m^2} \geq 1$, we obtain the bound~\cref{eq:page_rate_bounded-variance} as follows:
  \begin{align}
    \ex*{\norm*{\nabla f(\tilde x_K)}}
    &\leq
    \frac{2}{K}
    \prn*{
      c_\lambda m^{p-2} K
    }^{\frac{1}{p}}
    \prn*{
      2 \Delta \frac{K}{m^2}
      + 4 L \frac{K}{m^2} \prn*{\frac{4 L}{c_\lambda}}^{\frac{2}{p-2}}
      + \frac{4 \sigma^2 K}{L c_b^2 m^2}
    }^{\frac{p-1}{p}}\\
    &=
    \frac{2}{K}
    \prn*{
      c_\lambda m^{p-2} K
    }^{\frac{1}{p}}
    \prn*{\frac{K}{m^2}}^{\frac{p-1}{p}}
    \prn*{
      2 \Delta
      + 4 L \prn*{\frac{4 L}{c_\lambda}}^{\frac{2}{p-2}}
      + \frac{4 \sigma^2}{L c_b^2}
    }^{\frac{p-1}{p}}\\
    &=
    \frac{2 c_\lambda^{\frac{1}{p}}}{m}
    \prn*{
      2 \Delta
      + 4 L \prn*{\frac{4 L}{c_\lambda}}^{\frac{2}{p-2}}
      + \frac{4 \sigma^2}{L c_b^2}
    }^{\frac{p-1}{p}}
    =
    \O \prn*{K^{-1/2}}.
  \end{align}
  This rate implies that $K = \O(\epsilon^{-2})$ suffices to ensure $\ex*{\norm*{\nabla f(\tilde x_K)}} \leq \epsilon$.
  Hence, the oracle complexity is
  \begin{align}
    B + (K-1) \prn[\big]{
      \theta B + 2 (1 - \theta) b
    }
    =
    \O \prn*{b^2 + K b}
    =
    \O \prn*{
      K + c_b K^{3/2}
    }
    =
    \O \prn*{ \epsilon^{-3} },
  \end{align}
  which completes the proof.
\end{proof}

If $L$, $\Delta$, and $\sigma$ are known, setting $c_\lambda = \Theta \prn[\big]{\Delta^{\frac{2-p}{2}} L^{\frac{p}{2}}}$ and $c_b = \Theta \prn[\big]{\sigma / \sqrt{\Delta L}}$ in \cref{eq:page_rate_bounded-variance} gives
\begin{align}
  \ex*{\norm*{\nabla f(\tilde x_K)}}
  &\leq
  \O \prn[\bigg]{
    \frac{\sqrt{\Delta L}}{m}
  }
  =
  \O \prn[\bigg]{
    \sqrt{\frac{\Delta L}{K}}
    +
    \frac{\sigma}{K}
  }.
\end{align}
Hence, $K = \O \prn{\Delta L \epsilon^{-2} + \sigma \epsilon^{-1}}$ suffices to ensure $\ex*{\norm*{\nabla f(\tilde x_K)}} \leq \epsilon$.

For this choice of $K$, we also have $b = \O \prn{\sigma \epsilon^{-1}}$.
To see this, first suppose that $\Delta L \leq \sigma \epsilon$.
Then the iteration complexity reduces to $K = \O(\sigma \epsilon^{-1})$, and hence $c_b m \leq K = \O(\sigma \epsilon^{-1})$.
On the other hand, if $\Delta L > \sigma \epsilon$, then $K = \O(\Delta L \epsilon^{-2})$.
Combining this bound with $c_b = \Theta(\sigma / \sqrt{\Delta L})$ gives $c_b m \leq c_b \sqrt K = \O(\sigma \epsilon^{-1})$.
Therefore, $b = \ceil*{c_b m} = \O \prn{\sigma \epsilon^{-1}}$.

Consequently, the oracle complexity is
\begin{align}
  B + (K-1) \prn[\big]{
    \theta B + 2 (1 - \theta) b
  }
  =
  \O \prn*{b^2 + Kb}
  =
  \O \prn[\bigg]{
    \frac{\Delta L \sigma}{\epsilon^3} + \frac{\sigma^2}{\epsilon^2}
  }.
\end{align}
Thus, the parameter-tuned version recovers the optimal oracle complexity~\citep{arjevani2023lower}.

\paragraph{Related work.}
Early variance-reduced methods~\citep{reddi2016stochastic,allen2016variance} for the nonconvex finite-sum problem~\cref{eq:sgd_problem} achieved the complexity $\O(n + n^{2/3} \Delta L \epsilon^{-2})$ under the individual smoothness assumption, namely, the $L$-smoothness of each $f_i$.
\citet{fang2018spider} improved this bound to $\O(n + \sqrt{n} \Delta L \epsilon^{-2})$ under average smoothness (i.e., \cref{asm:page_average_smoothness}), which relaxes individual smoothness.
In the bounded-variance setting, they also established the complexity $\O(\Delta L \epsilon^{-2} (1 + \sigma \epsilon^{-1}))$ under average smoothness.
These finite-sum and bounded-variance bounds are known to be optimal~\citep{li2021page,arjevani2023lower}.
Many other variance-reduced methods with near-optimal or optimal complexity bounds have since been developed, including \citep{wang2019spiderboost,pham2020proxsarah,li2021page,cutkosky2019momentum,jiang2024adaptive,levy2021storm+,kavis2022adaptive}.
Among them, PAGE~\citep{li2021page} is particularly simple and admits a concise analysis \citep{li2021short}.
Parameter-free variance-reduced methods have also been studied~\citep{jiang2024adaptive,levy2021storm+,kavis2022adaptive}.
The method of~\citep{jiang2024adaptive} achieves the complexities $\O(n + \sqrt{n} \epsilon^{-2})$ and $\O(\epsilon^{-3})$, although its finite-sum bound relies on individual smoothness.
The method of~\citep{kavis2022adaptive} obtains the finite-sum bound $\tilde \O(n + \sqrt{n} \epsilon^{-2})$ under individual smoothness, and its tuned version recovers $\tilde \O(n + \sqrt{n} \Delta L \epsilon^{-2})$.
Our PAGE instantiation provides parameter-free guarantees under average smoothness while retaining the simplicity of the original PAGE method.
Its tuned version recovers the optimal problem-parameter dependence in both the finite-sum and bounded-variance settings.

\section{Conclusion}
\label{sec:conclusion}
We developed a systematic framework for constructing parameter-free algorithms for smooth nonconvex optimization.
The key idea is to use higher-order regularization, with a regularization exponent larger than the order of the model error.
This choice makes the method robust to misspecification of the regularization parameter and yields complexity guarantees without using line search, trust regions, or other acceptance tests.

We instantiated this principle for gradient descent, Newton's method, the Gauss--Newton method, SGD, and PAGE.
For these methods, the resulting algorithms achieve the optimal or best-known dependence on the target accuracy without requiring prior knowledge of problem-dependent parameters.
When such parameters are known up to constant factors, the proposed tuning also recover the optimal or best-known dependence on these parameters.
These results show that higher-order regularization provides a simple and unified mechanism for designing parameter-free model-based algorithms in both deterministic and stochastic settings.

Several directions remain for future work.
A natural next step is to incorporate acceleration.
Recent parameter-free accelerated methods use line search or related adaptive mechanisms~\citep{marumo2024parameter,marumo2025universal,xiong2026parameter}, and it would be interesting to develop accelerated variants that retain the acceptance-test-free nature of our approach.
Another direction is to extend the framework to broader classes of problems, such as constrained or nonsmooth optimization.

\section*{Statements and Declarations}
\paragraph{Funding.}
This work was partially supported by JSPS KAKENHI (23H03351 and 24K23853) and JST CREST (JPMJCR24Q2).

\paragraph{Competing Interests.}
The authors declare that they have no competing interests.

\paragraph{Data Availability.}
Data availability is not applicable to this article, as no datasets were generated or analyzed during the current study.

\bibliographystyle{abbrvnat_nodoi}
\bibliography{myrefs}

\end{document}